\documentclass[a4paper, draf, 10pt]{amsart}
\usepackage{amssymb,amsfonts, amsmath, amsthm}
\usepackage{tikz}
\usepackage[active]{srcltx}
\usepackage{datetime}\usdate

\newcommand{\supast}{{${^\ast}$}}
\newcommand{\ulb}{{\textup{(}}}
\newcommand{\urb}{{\textup{)}}}

\newcommand{\T}{{\mathbb{T}}}

\newcommand{\Z}{{\mathbb{Z}}}

\newcommand{\idmap}{{\textup{id}}}
\newcommand{\ev}{\textup{ev}}
\newcommand{\rest}[1]{{\restriction_{#1}}}
\newcommand{\Ker}{{\textup{Ker\,}}}

\newcommand{\Gdual}{{\widehat G}}

\newcommand{\suml}{{\sum_{l\in\Z}}}

\newcommand{\sumn}{{\sum_{n\in\Z}}}
\DeclareFontFamily{OT1}{pzc}{}
\DeclareFontShape{OT1}{pzc}{m}{it}{<-> s * [1.15] pzcmi7t}{}
\DeclareMathAlphabet{\mathscr}{OT1}{pzc}{m}{it}

\newcommand{\hull}{{\mathscr h}}
\newcommand{\kernel}{{\mathscr k}}
\newcommand{\Hull}{{\mathscr H\,}}
\newcommand{\Kernel}{{\mathscr K\,}}
\newcommand{\Zeroes}{{\mathscr Z\,}}
\newcommand{\Ideal}{{\mathcal I}}
\newcommand{\Idealtilde}{{\widetilde{\mathcal I}}}

\newcommand{\Hulldomain}{{\mathscr A}}
\newcommand{\Kerneldomain}{{\mathscr B}}

\newcommand{\Zeroesdomain}{{\mathscr A}}
\newcommand{\Idealdomain}{{\mathscr B}}

\newcommand{\topspace}{{X}}
\newcommand{\pt}{x}
\newcommand{\homeo}{{\sigma}}
\newcommand{\prodhomeo}{{\homeo\times\idmap_\T}}
\newcommand{\orbit}[1]{{\Z\cdot{#1}}}
\newcommand{\orbitclosure}[1]{{\overline{\Z\cdot{#1}}}}
\newcommand{\orbitclosures}{{\mathcal{OC}}}

\newcommand{\dynsysshort}{{\Sigma}}
\newcommand{\per}{{\textup{Per}}}
\newcommand{\fix}{{\textup{Fix}}}
\newcommand{\aperpoints}{{\textup{Aper}(\homeo)}}
\newcommand{\perpoints}{{\per(\homeo)}}
\newcommand{\perpnew}{{{\per}_p(\homeo)}}
\newcommand{\setofperpoints}{{S_{\textup{p}}}}
\newcommand{\setofaperpoints}{{S_{\textup{ap}}}}

\newcommand{\coeffalg}{{C(\topspace)}}
\newcommand{\lone}{{\ell^1(\dynsysshort)}}
\newcommand{\loneZ}{{\ell^1(\Z)}}

\newcommand{\cstar}{{C^\ast(\dynsysshort)}}

\newcommand{\fs}{{c_{00}(\dynsysshort)}}
\newcommand{\loneelement}[1]{{\sum_{#1} a_{#1}\delta^{#1}}}

\newcommand{\piaper}[1]{{P_{#1}}}
\newcommand{\piper}[2]{{P_{#1,#2}}}
\newcommand{\piperintersection}[1]{{Q_{#1}}}

\newcommand\Laurentseries[1]{{#1((\delta))}}
\newcommand\Laurentpolynomials[1]{{#1(\delta)}}

\newcommand{\Eone}{E}

\newcommand{\amap}{\alpha}
\newcommand{\bmap}{\beta}
\newcommand{\fixba}{\fix(\bmap\circ\amap)}
\newcommand{\fixab}{\fix(\amap\circ\bmap)}

\newcommand{\st}{\phi}

\newcommand{\Fourier}{{\,\mathfrak F}}

\newcommand{\XtimesT}{{\topspace\times\T}}
\newcommand{\setofroots}[3]{{\{{#1}\in\T : {#1}^{#2}={#3}\}}}
\newcommand{\standardsetofroots}{{\setofroots{\mu}{p}{\lambda}}}

\theoremstyle{plain}

\newtheorem{theorem}{Theorem}[section]
\newtheorem{proposition}[theorem]{Proposition}
\newtheorem{lemma}[theorem]{Lemma}
\newtheorem{corollary}[theorem]{Corollary}

\theoremstyle{definition}

\newtheorem{definition}[theorem]{Definition}
\newtheorem{example}[theorem]{Example}
\newtheorem{remark}[theorem]{Remark}
\newtheorem*{assumption}{Assumption}

\numberwithin{equation}{section}

\begin{document}

%%%%%%%%%%%%%%%%%%%%%%%%%%%%%%%%% BEGIN FRONTMATTER %%%%%%%%%%%%%%%%%%%%%%%%%%%%%%%%%%%%

\title[Noncommutative spectral synthesis]{Noncommutative spectral synthesis for the involutive Banach algebra associated with a topological dynamical system}

\author{Marcel de Jeu}
\address{Marcel de Jeu, Mathematical Institute, Leiden University, P.O.\ Box 9512, 2300 RA Leiden, The Netherlands}
\email{mdejeu@math.leidenuniv.nl}

\author{Jun Tomiyama}
\address{Jun Tomiyama, Department of Mathematics, Tokyo Metropolitan University, Minami-Osawa, Hachioji City, Japan}
\email{juntomi@med.email.ne.jp}



%\date{Version of \today,\,\,\currenttime\ hr}

\begin{abstract}
If $\dynsysshort=(\topspace,\homeo)$ is a topological dynamical system, where $\topspace$ is a compact Hausdorff space and $\homeo$ is a homeomorphism of $\topspace$, then a crossed product involutive Banach algebra $\lone$ is naturally associated with these data. If $\topspace$ consists of one point, then $\lone$ is the group algebra of the integers. In this paper, we study spectral synthesis for the closed ideals of $\lone$ in two versions, one modeled after $\coeffalg$, and one modeled after $\loneZ$. We identify the closed ideals which are equal to (what is the analogue of) the kernel of their hull, and determine when this holds for all closed ideals, i.e., when spectral synthesis holds. In both models, this is the case precisely when $\dynsysshort$ is free.
\end{abstract}

\maketitle

%%%%%%%%%%%%%%%%%%%%%%%%%%%%%%%%% END FRONTMATTER %%%%%%%%%%%%%%%%%%%%%%%%%%%%%%%%%%%%%%%

\section{Introduction}\label{sec:introduction}

Suppose $G$ is a locally compact abelian group, with dual group $\Gdual$. If $I\subset L^1(G)$ is a closed ideal, then its hull $\hull(I)\subset\Gdual$ is the closed set of common zeroes of the Fourier transforms of the elements of $I$, or, equivalently, the set of all maximal modular ideals of $L^1(G)$ containing $I$. If $S\subset\Gdual$ is closed, then its kernel $\kernel(S)$ is the closed ideal of $L^1(G)$ defined as $\kernel(S):=\bigcap_{\mathfrak m\in S}\mathfrak{m}$. It is a non-trivial fact (the regularity of $L^1(G)$) that $\hull\kernel(S)=S$ for all closed $S\subset\Gdual$, and a closed subset $S$ of $\Gdual$ is called a set of spectral synthesis if $\kernel(S)$ is the \emph{only} closed ideal $I$ of $L^1(G)$ such that $\hull(I)=S$. Determining sets of spectral synthesis is a delicate problem, and one may ask whether it is possible that \emph{all} closed subsets of $\Gdual$ are sets of spectral synthesis, in which case spectral synthesis is said to hold for $G$. This is asking for the injectivity of $\hull$ on the set of closed ideals, or, equivalently, requiring that each closed ideal $I$ is of the form $\kernel(S)$ for some closed subset $S$ of $\Gdual$ (which is then necessarily equal to $\hull(I)$). In that case, one can think of each $I$ as being synthesised, via intersection, from the maximal modular ideal---which one can regard as the evidently existing closed ideals of $L^1(G)$---in $\kernel(I)$. Alternatively, one can view each $I$ as being reconstructed from the set of common zeroes of the Fourier transforms of the elements of $I$. This problem was finally settled by Malliavin \cite{M}, \cite[Theorem 7.6.1]{R}: spectral synthesis holds for $G$ if and only if $G$ is compact. For compact $G$, all closed ideals of $L^1(G)$ are then self-adjoint, and the converse is also true, cf.\ \cite[Theorem~7.7.1]{R}. Hence all closed ideals of $L^1(G)$ are self-adjoint if and only if $G$ is compact.

Spectral synthesis has also been studied for other semisimple regular commutative Banach algebras than $L^1(G)$; see, for example, \cite[Chapter~ 8]{L}, or \cite[Chapter~5]{K}. Apart from $L^1(G)$, with $G$ compact and abelian, the only other common class of commutative Banach algebras for which spectral synthesis holds, and that we are aware of, are the algebras $C_0(\topspace)$, for a locally compact Hausdorff space $\topspace$. Passing to possibly noncommutative algebras, we note that for general $C^\ast$-algebras every closed ideal is the intersection of primitive ideals \cite{D}. Thus, here again every closed ideal can be synthesised from evidently present ideals, emerging from the representation theory of the algebra.

In this paper, we consider spectral synthesis for the involutive Banach algebra $\lone$ that is naturally associated with a dynamical system $\dynsysshort$, consisting of a compact Hausdorff space $\topspace$ and a homeomorphism $\homeo$ of $\topspace$. Its enveloping $C^*$-algebra $\cstar$ is well studied, but the investigation of this underlying involutive Banach algebra itself is of a more recent nature and was initiated in \cite{DST} and \cite{DT}. These algebras are considerably more complicated than their $C^*$-envelopes, as already becomes obvious from the case where $\topspace$ consists of one point. In that case, $\lone=\loneZ$, so that its $C^*$-envelope is $C(\T)$, and whereas $C(\T)$ has only self-adjoint closed ideals, $\loneZ$ also has non-self-adjoint closed ideals; and whereas spectral synthesis holds for $C(\T)$, it fails for $\loneZ$. For general $\topspace$, $\cstar$ naturally has only self-adjoint closed ideals, but it is known that this holds for the underlying algebra $\lone$ precisely when $\dynsysshort$ is free \cite[Theorem~4.4]{DST}, as an analogue of the result that $L^1(G)$ has only self-adjoint closed ideals precisely when $G$ is compact. Can we, then, also settle the matter of the validity of spectral synthesis for $\lone$, as it has been settled for $L^1(G)$?

In answering this question we have interpreted ``spectral synthesis'' as the reconstruction of a closed ideal of $\lone$ from a suitably defined set of common zeroes of functions associated to all the elements of the ideal, thus (as in the above examples) establishing that it belongs to an evidently existing family of ideals of $\lone$. There are two natural candidates to do this. Firstly, as in \cite{T3}, one can take $\coeffalg$ as a model, and ask which closed ideals are determined by the common zeroes of the coefficients of the elements of the ideal. Theorem~\ref{t:ideal_well_behaved} provides a number of answers, one of which is that these are precisely the ideals that are intersections of closed ideals naturally associated to orbit closures. Consequently, such ideals are intersections of primitive ideals, but the latter cannot be chosen freely. Spectral synthesis in this model then holds when every closed ideal can be so reconstructed, and Theorem~\ref{t:all_ideals_well_behaved} gives a number of equivalent conditions for this to hold, one of these being requiring $\dynsysshort$ to be free. Secondly, one can take $L^1(G)$, or in this case $\loneZ$, as a model, and consider the common zeroes of all (generalised) Fourier transforms of the elements of the closed ideal. In that case, Theorem~\ref{t:fixed_points_of_IZ} asserts that the reconstructible ideals are precisely the intersections of (freely chosen) primitive ideals naturally associated to the points of $\topspace$. According to Theorem~\ref{t:all_ideals_kernels}, spectral synthesis holds in this model precisely when $\dynsysshort$ is free again.

Thus the question of global spectral synthesis in both models has been settled. If it does not hold, then the study of spectral sets becomes relevant, but we leave that for future research.

\smallskip

To conclude this introduction, let us mention that these algebras $\lone$ are rather well accessible concrete examples of involutive Banach algebras, and with a rich variety of possible properties, depending on the dynamics. Major open issues are the question whether $\lone$ is always Hermitian, whether the closure of a proper ideal of $\lone$ in $\cstar$ is always proper (see Remark~\ref{r:C_closure_proper}), and whether each self-adjoint closed ideal is the kernel of an involutive representation of $\lone$. We hope to be able to report further on these algebras and their structure in the future.

\smallskip

This paper is organised as follows.

Section~\ref{sec:preliminaries} contains not only the necessary notions and notations, but also a detailed investigation of three families of closed ideals of $\lone$. These ideals play a key role in the subsequent sections.

Section~\ref{sec:function_space_model} is concerned with spectral synthesis in the $\coeffalg$-model. The main results are Theorem~\ref{t:ideal_well_behaved} (identifying the reconstructible ideals), Theorem~\ref{t:hull_kernel_setup_function_model} (formulating the properties of the analogues of the hull and kernel operators), and Theorem~\ref{t:all_ideals_well_behaved} (determining when spectral synthesis holds in this model).

Section~\ref{sec:l_one_model} takes up spectral synthesis in the $\loneZ$-model. The main results are now Theorem~\ref{t:fixed_points_of_IZ} (identifying the reconstructible ideals), Theorem~\ref{t:hull_kernel_setup_L_1_model} and  (formulating the properties of the analogies of the hull and kernel operator), and Theorem~\ref{t:all_ideals_kernels} (determining when spectral synthesis holds in this model).

Appendix~\ref{sec:appendix} contains the underlying abstract framework of hull-kernel-type operators. It is completely elementary, but we know of no reference for the basic properties of the combination of such operators, which we use in both Section~\ref{sec:function_space_model} and Section~\ref{sec:l_one_model}. By explicitly including them here we also hope to avoid any future mildly annoying verification of these elementary, but not entirely obvious, generalities in other examples.

\section{Preliminaries}\label{sec:preliminaries}

In this section, we establish the basic notations and introduce the algebra $\lone$, along with three families of ideals. In Section~\ref{sec:function_space_model} and $Section~\ref{sec:l_one_model}$, the ideals thus obtained will play a role similar to that of the maximal modular ideals in spectral synthesis for $L^1(G)$, where $G$ is a locally compact abelian group.

Throughout this paper, $\topspace$ is a non-empty compact Hausdorff space, and $\homeo:\topspace\to\topspace$ is a homeomorphism. If $\pt\in\topspace$, we write $\orbit{x}$ and $\orbitclosure{x}$ for its orbit and the closure of its orbit, respectively. We let $\aperpoints$ and $\perpoints$ denote the aperiodic and the periodic points of $\homeo$, respectively. For $p\geq 1$, let $\perpnew$ be the set of points of period $p$.

The involutive algebra of continuous complex-valued functions on $\topspace$ is denoted by $\coeffalg$, and we write $\alpha$ for be the involutive automorphism of $\coeffalg$ induced by $\homeo$ via $\alpha(f) := f \circ \homeo^{-1}$, for $f \in\coeffalg$.

If $S\subset\topspace$, then we let $\kernel (S)=\{f\in\topspace : f\rest{S}=0\}$ be its usual kernel. If $I$ is an ideal of $\coeffalg$, let $\hull (I)=\{x\in\topspace : f(x)=0\textup{ for all }f\in I\}$ be its usual hull. Then $\kernel(S)=\kernel(\bar S)$, $\hull(I)=\hull(\bar I)$, and $\hull$ and $\kernel$ are mutually inverse bijections between the set of closed ideals of $\coeffalg$ on the one hand, and the set of closed subsets of $\topspace$ on the other hand. Their restrictions are mutually inverse bijections between the set of $\alpha$-invariant closed ideals of $\coeffalg$ and the set of $\homeo$-invariant closed subsets of $\topspace$.

\subsection{$\mathbf{\lone}$ and behaviour of ideals}\label{subsec:lone_introduction}

Via $n \mapsto \alpha^n$, the integers act on $\coeffalg$. With $\Vert\cdot\Vert$ denoting the supremum norm on $\coeffalg$, we let
\[
\lone = \{a: \mathbb{Z} \to\coeffalg : \Vert a \Vert := \sum_{n} \Vert a(n)\Vert < \infty\}.
\]
We supply $\lone$ with the usual twisted convolution as multiplication,
\[
(aa^\prime) (n) := \sum_{k \in \Z} a(k) \cdot \alpha^k (a^\prime(n-k))\quad(a, a^\prime \in \lone),
\]
and define an involution by
\[
a^* (n) = \overline{\alpha^n (a(-n))}\quad (a\in\lone),
\]
so that it becomes a unital Banach $\sp\ast$-algebra with isometric involution. We let $\cstar$ denote its enveloping $C^\ast$-algebra. If $\topspace$ consists of one point, then $\lone$ is the group algebra $\loneZ$, and $\cstar$ can be identified with $C(\T)$.

A convenient way to work with $\lone$ is provided by the following. For $n,m \in \mathbb{Z}$, let
\begin{equation*}
  \chi_{\{n\}} (m) =
  \begin{cases}
    1 &\text{if }m =n;\\
    0 &\text{if }m \neq n,
  \end{cases}
\end{equation*}
where the constants denote the corresponding constant functions in $\coeffalg$. Then $\chi_{\{0\}}$ is the identity element of $\lone$. Let $\delta = \chi_{\{1\}}$; then $\chi_{\{-1\}}=\delta^{-1}=\delta^*$. If we put $\delta^0=\chi_{\{0\}}$, then $\delta^n = \chi_{\{n\}}$, for all $n \in \mathbb{Z}$. We may view $\coeffalg$ as a closed abelian $^*$-subalgebra of $\lone$, namely as $\{a_0 \delta^0 \, : \, a_0 \in \coeffalg\}$. If $a \in \lone$, and if we write $a(n)=a_n$ for short, then $a= \loneelement{n}$, and $\Vert a \Vert=\sum_n \Vert a_n\Vert<\infty$.
In the rest of this paper we will constantly use this series representation $a= \loneelement{n}$ of an arbitrary element $a\in\lone$, for uniquely determined $a_n\in\coeffalg$. Thus $\lone$ is generated as a unital Banach algebra by an isometrically isomorphic copy of $\coeffalg$ and the elements $\delta$ and $\delta^{-1}$, subject to the relation $\delta f \delta^{-1}=\alpha (f)= f\circ\homeo^{-1}$, for $f \in\coeffalg$. The isometric involution is determined by $f^*=\overline f$ ($f\in\coeffalg$), and $\delta^*=\delta^{-1}$. Hence the inner automorphism $\textup{Ad\,}\delta$ of $\lone$ is involutive, it leaves $\coeffalg$ invariant, and its restriction to $\coeffalg$ is $\alpha$.

Let $\fs$ denote the finitely supported elements of $\lone$. It is a dense involutive subalgebra.

\begin{definition}\label{d:primitive_ideal}
In this paper, a \emph{primitive ideal of $\lone$} is the kernel of a topologically irreducible unital involutive representation of $\lone$ on a Hilbert space.
\end{definition}

Note that this definition is not the one as used in, e.g., \cite{BD}, where a primitive ideal is defined in a purely algebraic fashion as the kernel of an abstract algebraically irreducible representation on an arbitrary complex vector space. Our definition is convenient for our purposes, to shorten terminology somewhat, and modeled after the situation for $C^*$-algebras, for which the primitive ideals as in \cite{BD} are precisely the ideals defined analogously to Definition~\ref{d:primitive_ideal} (see \cite{D} for this non-trivial fact).
The usual argument shows that an involutive representation of $\lone$ on a Hilbert space is automatically continuous, in fact even contractive. Hence a primitive ideal is a self-adjoint closed ideal.

If $L\subset\coeffalg$ is a linear subspace, let:
\begin{enumerate}
\item $\Laurentpolynomials{L}:=\{a=\loneelement{n}\in\fs : a_n\in L\textup{ for all }n\in\Z\}$;
\item $\Laurentseries{L}:=\{a=\loneelement{n}\in\lone : a_n\in L\textup{ for all }n\in\Z\}$.
\end{enumerate}

The following is readily verified.

\begin{lemma}\label{l:Laurent_properties} Let $L$ be a linear subspace of $\coeffalg$. Then:
\begin{enumerate}
\item $\overline{\Laurentpolynomials{L}}=\overline{\Laurentseries{L}}=\Laurentseries{\bar L}$;
\item $\Laurentpolynomials{L}$ is closed in $\lone$ if and only if $L=\{0\}$, and $\Laurentseries{L}$ is closed if and only if $L$ is closed.
\item $\Laurentpolynomials{L}$ is an ideal of $\fs$ if and only if $L$ is an $\alpha$-invariant ideal of $\coeffalg$;
\item $\Laurentseries{L}$ is a closed ideal of $\lone$ if and only if $L$ is an $\alpha$-invariant closed ideal of $\coeffalg$. In that case, $\Laurentseries{L}$ is self-adjoint;
\item If $\{L_\alpha : \alpha\in A\}$ is a collection of linear subspaces of $\coeffalg$, then:
    \begin{enumerate}
    \item $\sum_{\alpha\in A}\Laurentpolynomials{L_\alpha}=\Laurentpolynomials{(\sum_{\alpha\in A}L_\alpha )}$, and $\bigcap_{\alpha\in A}\Laurentpolynomials{L_\alpha}=\Laurentpolynomials{(\bigcap_{\alpha\in A}L_\alpha)}$;
    \item $\sum_{\alpha\in A}\Laurentseries{L_\alpha}\subset\Laurentseries{(\sum_{\alpha\in A}L_\alpha)}$, and $\bigcap_{\alpha\in A}\Laurentseries{L_\alpha}=\Laurentseries{(\bigcap_{\alpha\in A}L_\alpha)}$.
    \end{enumerate}
\end{enumerate}
\end{lemma}

As in \cite{T4} for $\cstar$, we distinguish three types of ideals in $\lone$. For their definition we use
the canonical involutive norm one projection $\Eone:\lone\to\coeffalg$, given by $\Eone(a)=a_0$, for $a=\loneelement{n}\in\lone$. The following properties are easy to check.

\begin{lemma}\label{l:basic_E_properties}Let $a=\loneelement{n}\in\lone$. Then:
\begin{enumerate}
\item $\Eone(f\cdot a\cdot g)=fg\Eone(a)\quad (f,g\in\coeffalg)$;
\item $\Eone(\delta\cdot a\cdot\delta^{-1})=\delta\cdot \Eone(a) \cdot\delta^{-1}=\alpha(\Eone(a))=\Eone(a)\circ\sigma^{-1}$;
\item $\Eone(a^\ast a)=\sum_{n}|a_n\circ\homeo^n|^2$;
\item $\Eone$ is injective on the positive cone of $\lone$;
\item If $I$ is an ideal of $\lone$, then:
\begin{enumerate}
\item $\Eone(I)$ is an $\alpha$-invariant ideal of $\coeffalg$;
\item $\Eone(I)=\{a_n : a=\loneelement{n}\in I\}$;
\item $I\subset\Laurentseries{\Eone(I)}$;
\item $\Eone(I)=\{0\}$ if and only if $I=\{0\}$.
\end{enumerate}
\end{enumerate}
If $\{L_\alpha : \alpha\in A\}$ is a collection of linear subspaces of $\lone$, then $\Eone(\sum_{\alpha\in A} L_\alpha)=\sum_{\alpha\in A}\Eone(L_\alpha)$.
\end{lemma}

\begin{definition} Let $I$ be an ideal of $\lone$. Then
\begin{enumerate}
\item $I$ is \emph{well behaved} if $\Eone(I) \subset I$;
\item $I$ is \emph{badly behaved} if $\Eone(I) = \coeffalg$;
\item $I$ is \emph{plain} if $E(I)\neq \coeffalg$ and $\Eone(I)\not\subset I$.
\end{enumerate}
\end{definition}

\begin{example}\label{e:lone_example_behaviour}
If $\topspace$ consists of one point, so that $\lone=\loneZ$, then there are no plain ideals. The well behaved ideals are $\{0\}$ and all ideals containing $c_{00}(\Z)$. The badly behaved ideals are all ideals containing $c_{00}(\Z)$, together with the badly behaved ideals which are not well behaved; the latter family admitting no explicit description. The picture simplifies when restricting our attention to closed ideals: there are no plain closed ideals in $\loneZ$, and the only well behaved closed ideals are $\{0\}$ and $\loneZ$. The badly behaved closed ideals of $\loneZ$ are precisely the non-zero closed ideals. We will see later (cf.\ Proposition~\ref{p:behaviour_of_families}) how plain self-adjoint closed ideals can sometimes be obtained for non-trivial $\topspace$.
\end{example}

The following Lemma follows easily from the definitions.

\begin{lemma}\label{l:behaved_lemma}
\begin{enumerate}
\item Arbitrary intersections and sums of well behaved ideals are well behaved.
\item If the ideal $I$ is badly behaved and $J\supset I$ for an ideal $J$, then $J$ is badly behaved.
\item An ideal $I$ is both well behaved and badly behaved if and only if it contains $\fs$; all other ideals fall into precisely one category.
\item The closed ideal $\lone$ is the only closed ideal that is both well behaved and badly behaved; all other closed ideals fall into precisely one category.
\end{enumerate}
\end{lemma}

\begin{lemma}\label{l:well_behaved_ideals} Let $I$ be an ideal of $\lone$. Then:
\begin{enumerate}
\item $I$ is well behaved if and only if $\Eone(I)=I\cap\coeffalg$;
\item If $I$ is well behaved and closed, then $\Eone(I)=I\cap\coeffalg$ is an $\alpha$-invariant closed ideal of $\coeffalg$.
\end{enumerate}
\end{lemma}

\begin{proof}
         Suppose that $I$ is a well behaved ideal. Then $E(I) \subset I \cap \coeffalg\subset\Eone(I)$.
Hence $\Eone(I) = I \cap \coeffalg$, and the converse is clear. The second part is now obvious.
\end{proof}

This leads to the following description of the well behaved closed ideals and their automatic self-adjointness.

\begin{corollary}\label{c:characterisation_well_behaved_closed_ideals}
\begin{enumerate}
\item The involutive norm one projection $\Eone$ induces a bijection $I\leftrightarrow\Eone(I)$ between the well behaved closed ideals of $\lone$ and the $\alpha$-invariant closed ideals of $\coeffalg$; the inverse map sends an $\alpha$-invariant closed ideal $I^\prime$ of $\coeffalg$ to $\Laurentseries{I^\prime}$.
\item All well behaved closed ideals of $\lone$ are self-adjoint.
\end{enumerate}
\end{corollary}

\begin{proof}
The routine proof of the first part is left to the reader. For the second, one need then merely note that all closed ideals of $\coeffalg$ are self-adjoint.
\end{proof}

Although we will not need it, the following result is worth noticing: a quotient of $\lone$ by a well behaved closed ideal is again an algebra in our class.

\begin{proposition}
Let $I$ be a well behaved closed ideal of $\lone$. Put $S_I=\hull(\Eone(I))$, so that $S_I$ is a $\homeo$-invariant closed subset of $\topspace$, and let $\homeo_I$ denote the restriction of $\homeo$ to $S_I$. Let $\dynsysshort_I=(S_I,\homeo_I)$ denote the resulting dynamical system, with associated algebra $\ell^1(\dynsysshort_I)$, generated by $C(S_I)$ and a unitary $\delta_I$. Then the map, sending $\sum_n a_n\delta^n\in\lone$ to $\sum_n a_n\rest{S_I}\delta_I^n$, is a contractive unital involutive Banach algebra homomorphism from $\lone$ onto $\ell^1(\dynsysshort_I)$, inducing an isometric involutive Banach algebra isomorphism between $\lone/I$ and $\ell^1(\dynsysshort_I)$.
\end{proposition}

\begin{proof}
We recall that a bounded surjective linear map $T:X\to Y$ between two Banach spaces is a quotient map, i.e., it induces an isometry between $X/\Ker T$ and $Y$, precisely when $T$ maps the open unit ball of $X$ onto that of $Y$. Now $\textup{Res}:\coeffalg\to C(S_I)$ is a surjective (by Tietze's Theorem) morphism of $C^\ast$-algebras. The induced map between $\coeffalg/\kernel(S_I)$ and $C(S_I)$ is then an isomorphism of $C^\ast$-algebras, hence automatically isometric. Therefore $\textup{Res}:\coeffalg\to C(S_I)$ is a quotient map. Since there are only countably many coefficients to be taken into account, it is now clear that the described map between $\lone$ and $\ell^1(\dynsysshort_I)$ is surjective, and in fact again a quotient map. As it is easily checked to be a unital involutive Banach algebra homomorphism, and its kernel is equal to $\Laurentseries{\kernel(S_I)}=\Laurentseries{\kernel\hull (\Eone(I)}=\Laurentseries{\Eone(I)}=I$, it induces an isometric involutive Banach algebra isomorphism between $\lone/I$ and $\ell^1(\dynsysshort_I)$.
\end{proof}

\subsection{Three families of ideals}\label{subsec:families_of_ideals}

We will now describe a number of irreducible involutive representations of $\lone$ that made an earlier appearance in \cite{DT}, and investigate their kernels. The ideals thus obtained will be an important ingredient in Sections~\ref{sec:function_space_model} and~\ref{sec:l_one_model}.

As a consequence of the general theory, cf.\ \cite{D}, there is bijection (via extension and restriction) between the pure states of $\lone$ and $\cstar$, and between the irreducible GNS-representations of the two algebras. Now, for each $\pt\in\topspace$, point evaluation is a pure state $\ev_\pt$ on $\coeffalg$, and all pure state extensions of $\ev_\pt$ to $\cstar$ therefore yield irreducible GNS-representations of $\cstar$, hence of $\lone$ by restriction. Since these pure state extensions of point evaluations to $\cstar$ (hence to $\lone)$ are well understood, as are their GNS-representations, we obtain explicitly given irreducible involutive representations of $\lone$. Referring to \cite[\S 4]{T1} for further details and proofs, the description is as follows.

First of all, if $x \in\aperpoints$, then there is a unique pure state extension of $\ev_x$ to $\lone$, which we denote by $\st_{{x}}$. The Hilbert space for the GNS-representation $\pi_{{x}}$ corresponding to $\st_{{x}}$ has an orthonormal basis $(e_k)_{k\in\Z}$, and the representation itself is determined by $\pi_{{x}}(\delta)e_k=e_{k+1}$, for $k\in\Z$, and $\pi_{{x}}(f)e_k=f(\homeo^k {x})e_k$, for $k\in\Z$. The vector $e_0$ reproduces the state $\st_{{x}}$ of $\cstar$.

If ${x}\in\perpoints$, say ${x}\in\perpnew$ $(p\geq 1)$, then the pure state extensions of $\ev_{{x}}$ to $\lone$ are in bijection with the points in $\T$, and we denote these pure states of $\lone$ by $\st_{{x},\lambda}$, for $\lambda\in\T$. The Hilbert space for the GNS-representation $\pi_{{x},\lambda}$ corresponding to $\st_{{x},\lambda}$ has an orthonormal basis $\{e_0,\ldots,e_{p-1}\}$, $\pi_{{x},\lambda}(\delta)$ is represented with respect to this basis by the matrix
\[
\left(
\begin{array}{ccccc}
0 & 0 & \ldots & 0 & \lambda \\
1 & 0 & \ldots & 0 & 0 \\
0 & 1 & \ldots & 0&0\\
\vdots & \vdots & \ddots &\vdots &\vdots \\
0 & 0 & \ldots & 1&0
\end{array}\right),
\]
and, for $f\in\coeffalg$, $\pi_{{x},\lambda}(f)$ is represented with respect to this basis by the matrix
\[
\left(
\begin{array}{cccc}
f({x}) & 0 & \ldots & 0 \\
0 & f (\homeo {x}) & \ldots & 0 \\
\vdots & \vdots & \ddots & \vdots \\
0 & 0 & \ldots & f(\homeo^{p-1} {x})
\end{array} \right).
\]
The vector $e_0$ reproduces the state $\st_{{x},\lambda}$ of $\lone$.

If $x\in\aperpoints$, then we write $\piaper{x}$ for the primitive ideal $\Ker \pi_x$ of $\lone$, and if $x\in\perpoints$ and $\lambda\in\T$, then the primitive ideal $\Ker \pi_{x,\lambda}$ is denoted by $\piper{x}{\lambda}$. If $x\in\perpoints$, we let $\piperintersection{x}=\bigcap_{\lambda\in\T}\piper{x}{\lambda}$. Note that these ideals $\piaper{x}$, $\piper{x}{\lambda}$ and $\piperintersection{x}$ are self-adjoint. If $\topspace$ consists of one point $\pt$, so that $\lone=\loneZ$, then only the second family $\piper{\pt}{\lambda}$ occurs, and $ \piper{\pt}{\lambda}=\{a\in\lone : \mathcal F (a)(\lambda)=0\}$, where $\mathcal F(a)(\lambda)=\sumn \lambda^n a_n$ $(\lambda\in\T)$ is the usual Fourier transform.\footnote{This definition is more convenient in our setup than the alternative $\mathcal F(a)(\lambda)=\sumn \lambda^{-n} a_n$.} Hence we retrieve the usual maximal modular ideals for $\loneZ$ and we have $\piperintersection{\pt}=\{0\}$ by the injectivity of the Fourier transform.

Unless $\topspace$ consists of one point, there exist unitary equivalences between members of the family $\{\pi_{x} : x\in\aperpoints\}\,\cup\,\{\pi_{x,\lambda} : x\in\perpoints,\lambda\in\T\}$. Hence the indices as used for these families of primitive ideals should not be thought of as a unique parametrisation, quite contrary to the case of $\loneZ$. In Remark~\ref{r:not_type_I_remark} we will give the precise relation between the indices, the unitary equivalence classes of involutive representations, and the primitive ideals.

The following description is the basis for further
 investigation of these ideals.

\begin{proposition}\label{p:belonging_to_kernel}Let $\loneelement{n}\in\lone$.
\begin{enumerate}
\item If $x\in\aperpoints$, then $a\in\piaper{x}$ if and only if $a_n\rest{\orbitclosure{x}}=0$, for all $n\in\Z$.
\item If $x\in\perpnew$ for some $p\geq 1$, and $\lambda\in\T$, then $a\in\piper{x}{\lambda}$ if and only if
\begin{equation}\label{e:vanishing}
\sum_{l\in\Z}\lambda^l a_{lp + j}(x^\prime) = 0,
\end{equation}
for all $j\in\{0,1,\ldots,p-1\}$ and all $x^\prime\in\orbit{x}$.
\item If $x\in\perpoints$, then $a\in\piperintersection{x}$ if and only if $a_n\rest{\orbit{x}}=0$, for all $n\in\Z$.
\end{enumerate}
\end{proposition}

\begin{proof}
The first part follows easily from the requirement $\pi_{x}(a)e_k=0$ $(k\in\Z)$, which is equivalent to $a\in\piaper{x}$. As to the second part, taking into account that $\pi_{x,\lambda}(\delta)^p=\lambda\cdot\idmap$ we see that $a\in\piper{x}{\lambda}$ if and only if
\begin{align*}
0&=\sumn \pi_{x,\lambda}(a_n)\pi_{x,\lambda}(\delta)^n e_k\\
&=\sum_{j=0}^{p-1}\suml \pi_{x,\lambda}(a_{lp+j})\pi_{x,\lambda}(\delta)^{lp+j} e_k\\
&=\sum_{j=0}^{p-1}\left[\suml \lambda^l \pi_{x,\lambda}(a_{lp+j})\right]\pi_{x,\lambda}(\delta)^j e_k,
\end{align*}
for all $k\in\{0,1,\ldots,p-1\}$. Since, for fixed $k$, the elements $\pi_{x,\lambda}(\delta)^j e_k$ $(j=0,\ldots,p-1)$ are, up to non-zero multiples, simply the basis vectors $\{e_0,\ldots,e_{p-1}\}$, and the action of $\coeffalg$ on this basis is diagonal, this holds if and only if
\[
\suml \lambda^l \pi_{x,\lambda}(a_{lp+j})\pi_{x,\lambda}(\delta)^j e_k=0,
\]
for all $j,k\in\{0,1,\ldots,p-1\}$. Applying $\pi_{x,\lambda}(\delta)^{-j}$ to this relation, we find the equivalent requirement
\[
\suml \lambda^l \pi_{x,\lambda}(a_{lp+j}\circ \homeo^j) e_k=0,
\]
for all $j,k\in\{0,1,\ldots,p-1\}$. Making the diagonal action of $\coeffalg$ explicit this translates into
\[
\suml \lambda^l a_{lp+j}(\homeo^{j+k}x)=0,
\]
for all $j,k\in\{0,1,\ldots,p-1\}$, which is the vanishing property on $\orbit{x}$ as stated in the second part of the Proposition.

Turning to the third part, let $p$ be the period of $x$. If $a_n\rest{\orbit{x}}=0$ for all $n\in\Z$, then part (2) implies that $a\in\piper{x}{\lambda}$ for all $\lambda\in\T$, hence $a\in\piperintersection{x}=\bigcap_{\lambda\in\T}\piper{x}{\lambda}$. Conversely, if $a\in\piper{x}{\lambda}$ for all $\lambda\in\T$, then \eqref{e:vanishing} holds for all $j\in\{0,1,\ldots,p-1\}$, all $x^\prime\in\orbit{x}$ and all $\lambda\in\T$. For fixed $j$ and $x^\prime$, the validity of \eqref{e:vanishing} for all $\lambda\in\T$ simply means that the map $l\mapsto a_{lp+j}(x^\prime)$, which is in $\loneZ$, has zero Fourier transform. Hence $a_{lp+j}(x^\prime)=0$ for all $l\in\Z$, $j\in\{0,1,\ldots,p-1\}$ and $x^\prime\in\orbit{x}$, which is an alternative way of expressing that $a_n\rest{\orbit{x}}=0$ for all $n\in\Z$.
\end{proof}

For the definition of $\piperintersection{x}$, countable intersection will do, as is implied by the next result.

\begin{corollary}\label{c:dense_subset_intersection} Let $x\in\perpoints$, and let $\T_x$ be a dense subset of $\T$. Then
\[
\piperintersection{x}=\bigcap_{\lambda\in\T_x}\piper{x}{\lambda}
\]
\end{corollary}

\begin{proof}
Certainly $\piperintersection{x}\subset\cap_{\lambda\in\T_x}\piper{x}{\lambda}$. For the reverse inclusion, if $a=\loneelement{n}\in\piper{x}{\lambda}$ for all $\lambda\in\T_x$, then the Fourier transform occurring in the conclusion of the proof of the third part of Proposition~\ref{p:belonging_to_kernel} is zero on $\T_x$, hence on $\T$, and it follows as before that $a_n\rest{\orbit{x}}=0$, for all $n\in\Z$. Hence $a\in\piperintersection{x}$.
\end{proof}

We will now collect a number of further consequences of Proposition~\ref{p:belonging_to_kernel}, and we start with the existence of well behaved, badly behaved and (possibly) plain self-adjoint closed ideals in $\lone$, as announced in the discussion following Lemma~\ref{l:behaved_lemma}. Note that for $\loneZ$ only the second and third part of Proposition~\ref{p:behaviour_of_families} are non-vacuous.

\begin{proposition}\label{p:behaviour_of_families}\quad
\begin{enumerate}
\item If $x\in\aperpoints$, then $\piaper{x}=\Laurentseries{\kernel(\orbitclosure{x})}$ is a well behaved proper \ulb self-adjoint\urb\ closed ideal.
\item If $x\in\perpoints$ and $\lambda\in\T$, then $\piper{x}{\lambda}$ is a badly behaved proper self-adjoint closed ideal.
\item If $x\in\perpoints$, then:
\begin{enumerate}
\item $\piperintersection{x}=\Laurentseries{\kernel(\orbit{x})}$ is a well behaved proper \ulb self-adjoint\urb\ closed ideal;
\item If $\lambda\in\T$, then $\piperintersection{x}$ is the largest well behaved ideal contained in $\piper{x}{\lambda}$.
\end{enumerate}
\item If $x_1\in\aperpoints$ and $x_2\in\perpoints\cap(\topspace\setminus\orbitclosure{x_1})$, then $\piaper{x_1}\cap\piper{x_2}{\lambda}$ is a plain (hence proper) self-adjoint closed ideal, for all $\lambda\in\T$.
\end{enumerate}
\end{proposition}

\begin{proof}
Parts (1) and (3)(a) are immediate from the first and third part of Proposition~\ref{p:belonging_to_kernel}, respectively. For part (3)(b), if $a=\loneelement{n}\in I$, where  $I$ is a well behaved ideal contained in $\piper{x}{\lambda}$, then $a_n\delta^n\in I\subset\piper{x}{\lambda}$, for all $n\in\Z$. Since \eqref{e:vanishing} then shows that $a_n\rest{\orbit{x}}=0$, for all $n\in\Z$, we see that $a\in\piperintersection{x}$. Hence $I\subset\piperintersection{x}$.

As to the second part, suppose $x\in\perpnew$. Now note that, for arbitrary $f\in\coeffalg$, $a:=f-(f/\lambda)\delta^p$ is in $\piper{x}{\lambda}$, since $\pi_{x,\lambda}(\delta^p)=\lambda\cdot\idmap$. Hence $f=\Eone(a)\in\Eone(\piper{x}{\lambda})$.

Turning to the fourth part, note that $\Eone( \piaper{x_1}\cap\piper{x_2}{\lambda})\subset \Eone(\piaper{x_1})=\{f\in\coeffalg : f\rest{\orbitclosure{x_1}}=0\}\neq\coeffalg$. Next, choose $f\in\coeffalg$ such that $f\rest{\orbitclosure{x_1}}=0$ and $f(x_2)\neq 0$. If $p$ is the period of $x_2$, let $a:=f-(f/\lambda)\delta^p$. Since $\pi_{x_1}(f)=0$ and $\pi_{x_2,\lambda}(\delta^p)=\lambda\cdot\idmap$, $a$ is in $\piaper{x_1}\cap\piper{x_2}{\lambda}$. However, by the second part of Proposition~\ref{p:belonging_to_kernel}, $\Eone(a)=f$ is not in $\piper{x_2}{\lambda}$ since $f(x_2)\neq 0$, hence it is certainly not in $\piaper{x_1}\cap\piper{x_2}{\lambda}$. Hence $\Eone(\piaper{x_1}\cap\piper{x_2}{\lambda})\not\subset \piaper{x_1}\cap\piper{x_2}{\lambda}$, and we see that $\piaper{x_1}\cap\piper{x_2}{\lambda}$ is plain.
\end{proof}

Another consequence of Proposition~\ref{p:belonging_to_kernel} is the following.

\begin{proposition}\label{p:separating_representations}
The family $\{\pi_x : x\in\aperpoints\}\,\cup\,\{\pi_{x,\lambda} : x\in\perpoints,\, \lambda\in\T\}$ of involutive representations of $\lone$ separates the elements of $\lone$.
\end{proposition}

\begin{proof}
This is clear from the first and third part of Proposition~\ref{p:behaviour_of_families}. Alternatively, we can remark that it must be the case, as it is even true for the superalgebra $\cstar$ of $\lone$, as a special case of \cite[Proposition~2]{T2}. The conceptual proof as given in \cite{T2} translates to a slightly easier one for $\lone$, as follows. If $a=\loneelement{n}\in\lone$ is in the kernel of all $\pi_x$ $(x\in\aperpoints)$ and $\pi_{x,\lambda}$ $(x\in\perpoints,\lambda\in\T)$, then certainly the states on $\lone$ used to define these representations vanish at $a^\ast a$. Since these states were taken to constitute \emph{all} pure state extensions of all states $\ev_x$ on $\coeffalg$, for $x\in\topspace$, the Krein-Milman theorem implies that all state extensions of all states $\ev_x$ $(x\in\topspace)$ vanish at $a^\ast a$. Now observe that, for all $x\in\topspace$, the map $a\mapsto a_0(x)$ is a state on $\lone$ extending $\ev_x$. The positivity follows from the fact that $(a^\ast a)_0=\sum_{n}|a_n\circ\homeo^n|^2$, and this also makes clear that, if $a^\ast a$ is in the simultaneous kernel of these states, then $a=0$.
\end{proof}

An argument similar to that in the proof of Proposition~\ref{p:separating_representations} shows that it is a priori clear that, if $a=\loneelement{n}$ is in $\piaper{x}$ or $\piperintersection{x}$, then $a_n(\homeo^n x)=0$ $(n\in\Z)$. Applying this to $a\delta^k$ $(k\in\Z)$, which is in the same ideal, we see that $\piaper{x}\subset\Laurentseries{\kernel(\orbitclosure{x})}$ $(x\in\aperpoints)$ and that $\piperintersection{x}\subset\Laurentseries{\kernel(\orbit{x})}$ $(x\in\perpoints)$. Since the reverse inclusion is clear from the description of the pertinent representations, one obtains an alternative proof of the first and third part of Proposition~\ref{p:behaviour_of_families}.

The following  is obvious from the first and third part of Proposition~\ref{p:behaviour_of_families}.

\begin{corollary}\label{c:orbit_closure_injectivemap}
Let $\orbitclosures$ be the set of orbit closures $\{\orbitclosure{x} : x\in X\}$. For $\orbitclosure{x}\in\orbitclosures$, let
\begin{equation*}
I(\orbitclosure{x})=
\begin{cases}
\piaper{x}\textup{ if }x\in\aperpoints ;\\
\piperintersection{x}\textup{ if }x\in\perpoints.
 \end{cases}
\end{equation*}
Then $I$ is a well defined inclusion reversing bijection between $\orbitclosures$ and the set $\{\piaper{x} : x\in \aperpoints\}\,\cup\,\{\piperintersection{x} : x \in\perpoints\}$, which consists of well behaved \ulb self-adjoint\urb\ closed ideals of $\lone$. It is given explicitly as $I(\orbitclosure{x})=\Laurentseries{\kernel(\orbitclosure{x})}$.
\end{corollary}

We conclude this section with complete information on the properness, triviality and all possible inclusions within and between these three families in our next result, followed by two of its consequences.

\begin{proposition}\label{p:properties_of_families}\quad
\begin{enumerate}
\item
\begin{enumerate}
\item The ideals $\piaper{x}$ $(x\in\aperpoints)$, $\piper{x}{\lambda}$ $(x\in\perpoints, \lambda\in\T)$, and $\piperintersection{x}$ $(x\in\perpoints)$ are proper self-adjoint closed ideals of $\lone$.
\item
\begin{enumerate}
\item For $x\in\aperpoints$, $\piaper{x} =\{0\}$ if and only if $\orbitclosure{x}=\topspace$.
    \item For $x\in\perpoints$ and $\lambda\in\T$, $\piper{x}{\lambda}\neq \{0\}$.
    \item For $x\in\perpoints$, $\piperintersection{x}=\{0\}$ if and only if $\orbit{x}=\topspace$.
\end{enumerate}
\end{enumerate}

\item The three sets $\{\piaper{x} : x\in\aperpoints\}$, $\{\piper{x}{\lambda} : x\in\perpoints,\,\lambda\in\T\}$ and $\{\piperintersection{x} : x\in\perpoints\}$ of proper self-adjoint closed ideals of $\lone$ are pairwise disjoint.
\item
\begin{enumerate}
\item For $x_1, x_2\in\aperpoints$, $\piaper{x_1}\subset\piaper{x_2}$ if and only if $\orbitclosure{x_1}\supset\orbitclosure{x_2}$.
\item For $x_1, x_2\in\perpoints$ and $\lambda_1,\lambda_2\in\T$, the following are equivalent;
    \begin{enumerate}
    \item $\piper{x_1}{\lambda_1}\subset\piper{x_2}{\lambda_2}$;
    \item $\orbit{x_1}=\orbit{x_2}$ and $\lambda_1=\lambda_2$;
    \item
            $\piper{x_1}{\lambda_1}=\piper{x_2}{\lambda_2}$.
    \end{enumerate}
\item For $x_1, x_2\in\perpoints$, the following are equivalent:
    \begin{enumerate}
    \item $\piperintersection{x_1}\subset\piperintersection{x_2}$;
    \item $\orbit{x_1}=\orbit{x_2}$;
    \item $\piperintersection{x_1}=\piperintersection{x_2}$.
    \end{enumerate}
\end{enumerate}
\item
\begin{enumerate}
\item Let $x_1\in\aperpoints$, $x_2\in\perpoints$, and $\lambda\in\T$. Then:
    \begin{enumerate}
    \item $\piaper{x_1}\not\supset\piper{x_2}{\lambda}$;
    \item $\piaper{x_1}\subset\piper{x_2}{\lambda}$ if and only if $\orbitclosure{x_1}\supset\orbit{x_2}$.
    \end{enumerate}
\item Let $x_1\in\aperpoints$ and $x_2\in\perpoints$. Then:
    \begin{enumerate}
    \item $\piaper{x_1}\not\supset\piperintersection{x_2}$;
    \item $\piaper{x_1}\subset\piperintersection{x_2}$ if and only if $\orbitclosure{x_1}\supset\orbit{x_2}$.
    \end{enumerate}
\item Let $x_1, x_2\in\perpoints$ and $\lambda\in\T$. Then:
\begin{enumerate}
\item $\piperintersection{x_1}\not\supset\piper{x_2}{\lambda}$;
\item $\piperintersection{x_1}\subset \piper{x_2}{\lambda}$ if and only if $\orbit{x_1}=\orbit{x_2}$.
\end{enumerate}
\end{enumerate}

\end{enumerate}
\end{proposition}

\begin{proof}
Part (1)(a) is obvious. Part (1)(b)(i) and (1)(b)(iii) follow from Proposition~\ref{p:behaviour_of_families}; part (1)(b)(ii) is clear since $\pi_{x,\lambda}$ is a finite dimensional representation of the infinite dimensional algebra $\lone$.

As to part (2), the first and third part of Proposition~\ref{p:belonging_to_kernel} imply that there is no overlap between the two families $\{\piaper{x} : x\in\aperpoints\}$ and $\{\piperintersection{x} : x\in\perpoints\}$. That the remaining two intersections are empty follows from Proposition~\ref{p:behaviour_of_families}: since the ideals $\piaper{x}$ $(x\in\aperpoints)$ and $\piperintersection{x}$ $(x\in\perpoints$ are well behaved proper closed ideals, and the ideals $\piper{x}{\lambda}$ $(x\in\perpoints,\lambda\in\T)$ are badly behaved proper closed ideals, there can be no overlap in view of part (4) of Lemma~\ref{l:behaved_lemma}.

Turning to part (3), part (3)(a) and (3)(c) are obvious from Proposition~\ref{p:belonging_to_kernel}. As to (3)(b), suppose that $\piper{x_1}{\lambda_1}\subset\piper{x_2}{\lambda_2}$. If $\orbit{x_1}\neq\orbit{x_2}$, then there exists $f\in\coeffalg$ such that $f\rest{\orbit{x_1}}=0$ and $f\rest{\orbit{x_2}}\neq 0$. The description of $\pi_{x_1,\lambda_1}$ and $\pi_{x_2,\lambda_2}$ then makes it clear that $f\in\piper{x_1}{\lambda_1}$, but $f\notin\piper{x_2}{\lambda_2}$. Hence the orbits must coincide, and we let $p$ be the period of $x_1$ and $x_2$. Since $1-(1/\lambda_1)\delta^p\in\piper{x_1}{\lambda_1}$, it is in $\piper{x_2}{\lambda_2}$ and applying $\pi_{x_2,\lambda_2}$ yields $1-(\lambda_2/\lambda_1)=0$, hence $\lambda_2=\lambda_1$. Therefore (3)(b)(i) implies (3)(b)(ii). It is clear from the second part of Proposition~\ref{p:belonging_to_kernel} that (3)(b)(ii) implies (3)(b)(iii), and the remaining implication in (3)(b) is trivial.

For part (4)(a)(i), if $\piaper{x_1}\supset\piper{x_2}{\lambda}$, then the kernel of the infinite dimensional irreducible representation $\pi_{x_1}$ would contain an ideal of the algebra of finite codimension in the algebra, which is impossible. As to (4)(a)(ii), assume that $\piaper{x_1}\subset\piper{x_2}{\lambda}$. Since the first part of Proposition~\ref{p:belonging_to_kernel} shows that $\kernel(\orbitclosure{x_1})\subset\piaper{x_1}$, we have $\kernel(\orbitclosure{x_1})\subset\piper{x_2}{\lambda}$. An application of the second part of Proposition~\ref{p:behaviour_of_families} with $j=0$, or an appeal to the description of $\pi_{x_2,\lambda}$, then implies that $\kernel(\orbitclosure{x_1})\subset\kernel(\orbit{x_2})$. Hence $\orbitclosure{x_1}\supset\orbit{x_2}$. The converse implication in (4)(a)(ii) is immediate from the first and second part of Proposition~\ref{p:belonging_to_kernel}.

Part (4)(b) is immediate from the first and third part of Proposition~\ref{p:belonging_to_kernel}.

For (4)(c)(i), if $\piperintersection{x_1}\supset\piper{x_2}{\lambda}$, then, using that part (2) of Proposition~\ref{p:behaviour_of_families} shows that $\piper{x}{\lambda}$ is badly behaved, the second part of Lemma~\ref{l:behaved_lemma} implies that $\piperintersection{x_1}$ is likewise badly behaved. This yields a contradiction between part (4) of Lemma~\ref{l:behaved_lemma} and part (3)(a) of Proposition~\ref{p:behaviour_of_families}. The proof of (4)(c)(ii)  is similar to that of (4)(a)(ii).
\end{proof}

Naturally, part (2) of Proposition~\ref{p:properties_of_families} also follows from part (4).

\begin{remark}\label{r:not_type_I_remark}\quad
\begin{enumerate}
\item It follows from \cite[Corollary~4.1.4]{tomiyama_book} that, among the set $\{\pi_{x} : x\in\aperpoints\}\,\cup\,\{\pi_{x,\lambda} : x\in\perpoints,\lambda\in\T\}$ of irreducible involutive representations of $\lone$, unitary equivalence occurs precisely between $\pi_{x_1}$ and $\pi_{x_1}$ for $x_1,x_2\in\aperpoints$ in the same orbit, and between $\pi_{x_1,\lambda}$ and $\pi_{x_2,\lambda}$ for $x_1,x_2\in\perpoints$ in the same orbit. Certainly the corresponding primitive ideals are then equal, but the converse is not true in general. As Proposition~\ref{p:properties_of_families} shows, if $x_0\in\perpoints,\,\lambda_0\in\T$, it is still true that the only representations in $\{\pi_{x} : x\in\aperpoints\} \,\cup\,\{\pi_{x,\lambda} : x\in\perpoints,\lambda\in\T\}$ with $\piper{x_0}{\lambda_0}$ as primitive ideal are the $\pi_{x^\prime,\lambda_0}$ with $x^\prime\in\orbit{x_0}$, i.e., precisely the involutive representations unitarily equivalent to $\pi_{x_0,\lambda_0}$. For $x_0\in\aperpoints$ this need not hold: the representations in $\{\pi_{x} : x\in\aperpoints\} \,\cup\,\{\pi_{x,\lambda} : x\in\perpoints, \lambda\in\T\}$ with $\piaper{x_0}$ as primitive ideal are precisely the $\pi_{x^\prime}$ with $x^\prime\in\aperpoints$ such that $\orbitclosure{x^\prime}=\orbitclosure{x_0}$, and it is possible that the set of such $x^\prime$ (the quasi-orbit of $x$) is strictly larger than $\orbit{x_0}$.
\item If $\topspace$ is metrizable, then \cite[Theorem~7.7]{tomiyama_notes_two} gives a number of equivalent conditions for the property that each irreducible representations of $\cstar$ is uniquely determined, up to unitary equivalence, by its primitive ideal. One of these is that the Birkhoff center $c(\homeo)$ of $\dynsysshort$ coincides with $\perpoints$, and another is that all irreducible representations of $\cstar$ are unitarily equivalent with the representations arising from pure state extensions of point evaluations as above (see \cite[Proposition~7.5]{tomiyama_notes_two} for an explicit counterexample if $c(\homeo)\supsetneqq\perpoints$).

It is tempting to try to deduce, from the known result for $\cstar$, that the analogous three properties are, for metrizable $\topspace$, also equivalent for $\lone$. It follows obviously from the result for $\cstar$ that each irreducible involutive representation of $\lone$ is unitarily equivalent with an irreducible involutive representation, arising from a pure state extension of a point evaluations as above, precisely when $c(\homeo)=\perpoints$. The question is harder, however, whether an involutive representation of $\lone$ is then also uniquely determined, up to unitary equivalence, by its primitive ideal. The obstacle (if it should be true) for the obvious ```proof'' is that, while the involutive representations of $\lone$ and $\cstar$ are in natural bijection, the relation is not so clear for primitive ideals: If two irreducible involutive representations of $\lone$ with the same kernel are extended to irreducible involutive representations of $\cstar$, then there is no obvious reason why these extended irreducible involutive representations should have the same kernel in $\cstar$.
\end{enumerate}
\end{remark}

We collect a number of consequences of Proposition~\ref{p:properties_of_families}; the first follows by inspection.

\begin{corollary}\label{c:possible_proper_inclusions} Let $I,J\in
\{\piaper{x} : x\in\aperpoints\}\,\cup\,\{\piper{x}{\lambda} : x\in\perpoints,\,\lambda\in\T\}\,\cup\,\{\piperintersection{x} : x\in\perpoints\}$. Then the only possible proper inclusions $I\subsetneqq J$ are the following.
\begin{enumerate}
\item $\piaper{x_1}\subsetneqq\piaper{x_2}$ $(x_1,x_2\in\aperpoints)$: this holds if and only if $\orbitclosure{x_1}\supsetneqq\orbitclosure{x_2}$;
\item $\piaper{x_1}\subsetneqq\piper{x_2}{\lambda}$ $(x_1\in\aperpoints,\, x_2\in\perpoints,\lambda\in\T)$: this holds if and only if $\orbitclosure{x_1}\supset\orbit{x_2}$;
\item $\piaper{x_1}\subsetneqq\piperintersection{x_2}$ $(x_1\in\aperpoints,\, x_2\in\perpoints)$: this holds if and only if $\orbitclosure{x_1}\supset\orbit{x_2}$;
\item $\piperintersection{x_1}\subsetneqq\piper{x_2}{\lambda}$ $(x_1,x_2\in\perpoints,\lambda\in\T)$: this holds if and only if $\orbit{x_1}=\orbit{x_2}$.
\end{enumerate}
\end{corollary}

Part (4)(a) and (4)(b) of Proposition~\ref{p:properties_of_families} imply the following.

\begin{corollary}\label{c:one_lambda_implies_all}
Let $x_1\in\aperpoints$ and $x_2\in\perpoints$. Then the following are equivalent:
\begin{enumerate}
\item There exists $\lambda\in\T$ such that $\piaper{x_1}\subset\piper{x_2}{\lambda}$;
\item $\orbitclosure{x_1}\supset\orbit{x_2}$;
\item $\piaper{x_1}\subset\piperintersection{x_2}=\bigcap_{\lambda\in\T} \piper{x_2}{\lambda}$.
\end{enumerate}
\end{corollary}

\section{Spectral synthesis: $\coeffalg$-model}\label{sec:function_space_model}

In this section we introduce the noncommutative hull and kernel for ideals of $\lone$, modeled after $\coeffalg$ in a manner analogous to that in \cite{T3}, and study the problem of spectral synthesis.

\begin{definition}
For for a linear subspace $I$ of $\lone$ define its noncommutative hull, $\Hull(I)$, as
\[
\Hull(I) = \{ x\in \topspace :  a_n(x) = 0\textup{ for all } a=\loneelement{n}\in I\textup{ and all } n\in\Z\},
\]
and for a subset $S$ of $\topspace$ define its noncommutative kernel, $\Kernel(S)$, as
\[
\Kernel(S) = \{ a=\loneelement{n} \in \lone : a_n\rest S = 0 \textup{ for all }n\in\Z\},
\]
with the usual convention that $\Kernel(\emptyset)=\lone$.
\end{definition}

The following two results are routinely verified, using part (5)(b) of Lemma~\ref{l:basic_E_properties} for the penultimate statement in Lemma~
\ref{l:basic_Hull_properties}.

\begin{lemma}\label{l:basic_Hull_properties}
Let $I$ be a linear subspace of $\lone$. Then:
\begin{enumerate}
\item $\Hull(I)$ is a closed subset of $\topspace$;
\item If $I^\prime$ is a linear subspace of $\lone$, and $I^\prime\subset I$, then $\Hull(I^\prime)\supset\Hull(I)$.
\item $\Hull(I)=\Hull(\bar I)$;
\item $\Hull(I)=\topspace$ if and only if $I=\{0\}$.
\end{enumerate}
If $\{I_\alpha : \alpha\in A\}$ is a collection of linear subspaces of $\lone$, then:
\begin{enumerate}
\item[(5)] $\Hull(\sum_{\alpha\in A} I_\alpha)=\bigcap_{\alpha\in A}\Hull(I_\alpha)$;
\item[(6)] $\Hull(\bigcap_{\alpha\in A} I_\alpha) \supset \bigcup_{\alpha\in A}\Hull(I_\alpha)$.
\end{enumerate}
If $I$ is an ideal of $\lone$, then:
\begin{enumerate}
\item[(7)] $\Hull(I)$ is a $\homeo$-invariant closed subset of $\topspace$;
\item[(8)] $\Hull(I)=\hull(\Eone(I))$;
\item[(9)] $\Hull(I)=\emptyset$ if and only if $\overline{\Eone(I)}=\coeffalg$.
\end{enumerate}
\end{lemma}

\begin{lemma}\label{l:basic_Kernel_properties} Let $S\subset\topspace$. Then:
\begin{enumerate}
\item $\Kernel(S)=\Laurentseries{\kernel(S)}$ is a closed $\coeffalg$-subbimodule of $\lone$, which is right invariant under $\delta$ and $\delta^{-1}$;
\item If $S^\prime\subset S$, then $\Kernel(S^\prime)\supset\Kernel(S)$;
\item $\Kernel(S)=\Kernel(\bar S)$;
\item If $S$ is $\homeo$-invariant, then $\Kernel(S)$ is a well behaved \ulb self-adjoint\urb\ closed ideal of $\lone$, and $\Eone(\Kernel(S))=\kernel(S)$;
\item $\Kernel(S)=\lone$ if and only if $S=\emptyset$, and $\Kernel(S)=\{0\}$ if and only if $\bar S=\topspace$;
\end{enumerate}
If $\{S_\alpha : \alpha\in A\}$ is a collection of subsets of $\topspace$, then:
\begin{enumerate}
\item[(6)] $\Kernel(\bigcup_{\alpha\in A}S_\alpha)=\bigcap_{\alpha\in A}\Kernel(S_\alpha)$;
\item[(7)] $\Kernel(\bigcap_{\alpha\in A}S_\alpha)\supset\sum_{\alpha\in A}\Kernel(S_\alpha)$.
\end{enumerate}
\end{lemma}

\begin{lemma}\label{l:Hull_Kernel_composition}\quad
\begin{enumerate}
\item If $S \subset\topspace$ is $\sigma$-invariant, then $\Hull\Kernel(S)=\hull\kernel(S)=\bar S$;
\item If $I$ is an ideal of $\lone$, then $\Kernel\Hull(I)=\Laurentseries{\overline{\Eone(I)}}=\overline{\Laurentseries{\Eone(I)}}\supset I$.
\end{enumerate}
\end{lemma}

\begin{proof}
For the first part, an application of Lemma~\ref{l:basic_Hull_properties}, Lemma~\ref{l:basic_Kernel_properties}, and the first part of Corollary~\ref{c:characterisation_well_behaved_closed_ideals} shows that
\begin{align*}
\Hull\Kernel(S)&=\Hull\left[\Laurentseries{\kernel(S)}\right]\\
&=\hull\left[\Eone[\Laurentseries{\kernel(S)}]\right]\\
&=\hull\kernel(S)\\
&=\bar S.
\end{align*}
For the second part, using Lemma~\ref{l:basic_Hull_properties}, Lemma~\ref{l:basic_Kernel_properties}, and Lemma~\ref{l:Laurent_properties}, we have
\begin{align*}
\Kernel\Hull(I)&=\Kernel\left[\hull(\Eone(I))\right]\\
&=\Laurentseries{\left[\kernel\hull(\Eone(I))\right]}\\
&=\Laurentseries{\overline{\Eone(I)}}\\
&=\overline{\Laurentseries{\Eone(I)}}\\
&\supset\Laurentseries{\Eone(I)}\\
&\supset I.
\end{align*}
\end{proof}

\begin{corollary}\label{c:fixed_Kernel_Hull_is_well_behaved}
Let $I$ be an ideal of $\lone$. Then $\Kernel\Hull(I)=I$ if an only if $I$ is a well behaved \ulb self-adjoint\urb\ closed ideal.
\end{corollary}

\begin{proof}
If $\Kernel\Hull(I)=I$, then it is clear from Lemma~\ref{l:basic_Kernel_properties} that $I$ is well behaved, (self-adjoint) and closed. Conversely, if $I$ is well behaved and closed, then the second part of Lemma~\ref{l:Hull_Kernel_composition} and the first part of Corollary~\ref{c:characterisation_well_behaved_closed_ideals} show that $\Kernel\Hull(I)=\overline{\Laurentseries{\Eone(I)}}=\overline{I}=I$.
\end{proof}

It is now possible to give a number of alternative descriptions of well behaved closed ideals in Theorem~\ref{t:ideal_well_behaved} below, reminiscent of similar results for a well behaved closed ideal of $\cstar$ (\cite[Theorem~2]{T2}). For part of the formulation we recall that the dual action $\alpha$ of $\T$ on $\lone$ is the strongly continuous representation $\lambda\mapsto\alpha_\lambda$ ($\lambda\in\T$) of $\T$ as isometric involutive automorphisms of $\lone$ determined by
\begin{equation*}\label{e:dual_action_definition}
\alpha_\lambda(f) = f\quad (f \in \coeffalg),\quad \alpha _\lambda( \delta^n) = \lambda^n \delta^n\quad(n\in\Z),
\end{equation*}
for $\lambda\in\T$. Hence, if $a=\loneelement{n}\in\lone$, and $\lambda\in\T$, then $\alpha_\lambda(a) = \sum_n \lambda^n a_n\delta^n$. Therefore the relation
\begin{equation}\label{e:dual_action_formula}
E(a) = \int_T \alpha_\lambda(a)\,d\lambda,
 \end{equation}
which needs some proof in the case of $\cstar$, is rather obvious for $\lone$.

\begin{theorem}\label{t:ideal_well_behaved}
Let $I$ be a closed ideal of $\lone$. Then the following are equivalent:
\begin{enumerate}
\item $I$ is well behaved;
\item $\Eone(I) = I \cap \coeffalg$;
\item $I = \Laurentseries{\Eone(I)}$;
\item $I = \Kernel\Hull(I)$;
\item If $I^\prime$ is a closed ideal of $\lone$, then $I^\prime\subset I$ if and only if $\Hull(I^\prime)\supset\Hull(I)$;
\item $I$ is invariant under the dual action of $\T$ on $\lone$;
\item There exist $\setofaperpoints\subset\aperpoints$ and $\setofperpoints\subset\perpoints$ such that
\begin{equation*}
I=\bigcap_{x\in\setofaperpoints}\piaper{x}\,\bigcap_{x\in\setofperpoints}\piperintersection{x}.
\end{equation*}
\end{enumerate}
In that case, if $\setofaperpoints\subset\aperpoints$ and $\setofperpoints\subset\perpoints$ are chosen such that
\begin{equation}\label{e:union_of_orbit_closures}
\hull(\Eone(I))=\overline{\bigcup_{x\in\setofaperpoints}\orbit{x}\,\cup\,\bigcup_{x\in\setofperpoints}\orbit{x}},
\end{equation}
giving the $\homeo$-invariant closed subset $\hull(\Eone(I))$ as the closure of a union of orbits, then
\begin{equation}\label{e:explicit_intersection_of_orbitclosure_ideals}
I=\bigcap_{x\in\setofaperpoints}\piaper{x}\,\bigcap_{x\in\setofperpoints}\piperintersection{x}
\end{equation}
establishes $I$ explicitly as an intersection as in part \ulb \textup{6}\urb. If, furthermore, $\T_x$ is dense in $\T$, for all $x\in\setofperpoints$, then
\begin{equation}\label{e:explicit_intersection_of_primitive_ideals}
I = \bigcap_{x\in\setofaperpoints}\piaper{x} \bigcap_{x\in\setofperpoints}\bigcap_{\lambda\in\T_x}\piper{x}{\lambda},
\end{equation}
establishes $I$ as an intersection of primitive ideals corresponding to pure state extensions of evaluations in points in  $\setofaperpoints\cup\setofperpoints$.
\end{theorem}

\begin{proof}
The equivalence of (1) and (2) is the first part of Lemma~\ref{l:well_behaved_ideals}. Corollary~\ref{c:characterisation_well_behaved_closed_ideals} shows that (1) implies (3). If (3) holds, then part (2) of Lemma~\ref{l:Laurent_properties} and the fact that $I$ is closed imply that $\Eone(I)$ is closed, hence Corollary~\ref{c:characterisation_well_behaved_closed_ideals} shows that $I$ is well behaved. Hence (1) and (3) are equivalent, and Corollary~\ref{c:fixed_Kernel_Hull_is_well_behaved} shows that (1) and (4) are equivalent. The equivalence of (4) and (5) is Lemma~\ref{l:order_inflection}. If (6) holds, then \eqref{e:dual_action_formula} makes it clear that $\Eone(I)\subset I$, hence $I$ is well behaved, and (6) implies (1). If (3) holds, then the definition of the dual action shows that (6) holds as well. Since all ideals in the intersection in (7) are well behaved by Proposition~\ref{p:behaviour_of_families}, their intersection is likewise a well behaved ideal by virtue of Lemma~\ref{l:behaved_lemma}. Hence (7) implies (1), and the proof will be finished once we establish that (3) implies (7). While doing so, we will establish the remaining statements as well.

Assume, then, that (3) holds. Since $\Eone(I)$ is an $\alpha$-invariant ideal of $\coeffalg$, $\hull(\Eone(I))$ is a closed $\homeo$-invariant subset of $\topspace$. Consequently, a choice of $\setofaperpoints\subset\aperpoints$ and $\setofperpoints\subset\perpoints$ such that \eqref{e:union_of_orbit_closures} holds is certainly possible. As already observed, the validity of (3) implies that $\Eone(I)$ is closed. Hence $\Eone(I)=\kernel\hull(\Eone(I))$, yielding
\begin{equation*}\label{e:Eone_I_as_intersection}
\Eone(I)=\bigcap_{x\in\setofaperpoints}\kernel(\orbitclosure{x})\,\bigcap_{x\in\setofperpoints}\kernel(\orbit{x}).
\end{equation*}
Since $I=\Laurentseries{\Eone(I)}$ by assumption, part (5)(b) of Lemma~\ref{l:Laurent_properties} shows that
\[
I=\bigcap_{x\in\setofaperpoints}\Laurentseries{\kernel(\orbitclosure{x})}\,\bigcap_{x\in\setofperpoints}\Laurentseries{\kernel(\orbit{x})}
\]

An appeal to the first and third part of Proposition~\ref{p:behaviour_of_families} then shows that \eqref{e:explicit_intersection_of_orbitclosure_ideals} holds, and Corollary~\ref{c:dense_subset_intersection} transforms \eqref{e:explicit_intersection_of_orbitclosure_ideals} into \eqref{e:explicit_intersection_of_primitive_ideals}.
\end{proof}

We will now consider spectral synthesis in the current model, i.e., investigate the extent to which the operators $\Hull$, $\Kernel$ are mutually inverse. We will make use of the generalities in Appendix~\ref{sec:appendix}, since the combination of Lemma~\ref{l:basic_Hull_properties}, Lemma~\ref{l:basic_Kernel_properties}, and Lemma~\ref{l:Hull_Kernel_composition} shows that we are in the context of Appendix~\ref{sec:appendix}, as is the content of the second sentence of Theorem~\ref{t:hull_kernel_setup_function_model}. Since Corollary~\ref{c:fixed_Kernel_Hull_is_well_behaved} describes the fixed points of $\Kernel\circ\Hull$, Corollary~\ref{c:corollary_of_three_maps_lemma} and Lemma~\ref{l:min_max_lemma} then imply the remaining statements.

Note that part (1), (3) and (4) are valid regardless of the dynamics, and that the validity or failure of spectral synthesis is considered in (2).

\begin{theorem}\label{t:hull_kernel_setup_function_model}
Let $\Hulldomain$ be the set of all closed ideals of $\lone$ and let $\Kerneldomain$ be the set of all $\homeo$-invariant closed subsets of $\topspace$, both ordered by inclusion. Then $\Hull: \Hulldomain\to \Kerneldomain$ and $\Kernel: \Kerneldomain\to \Hulldomain$ are decreasing, $\Kernel\circ\Hull(I)\succ I$ for all $I\in\Hulldomain$, and $\Hull\circ\Kernel=\idmap_{\Kerneldomain}$. Let ${\Hulldomain}_{wb}$ be the set of well behaved closed ideals of $\lone$. Then:
\begin{enumerate}
\item $\Hull:{\Hulldomain}_{wb}\to\Kerneldomain$ and $\Kernel:\Kerneldomain\to {\Hulldomain}_{wb}$ are mutually inverse bijections;
\item The following are equivalent:
\begin{enumerate}
\item $\Hull$ is injective on $\Hulldomain$;
\item Each closed ideal $I$ of $\lone$ is of the form $\Kernel(S)$ for a $\homeo$-invariant closed subset $S$ of $\topspace$;
\item Each closed ideal of $\lone$ is well behaved;
\item For each $\homeo$-invariant closed subset $S$ of $\topspace$, $\Kernel(S)$ is the unique closed ideal $I^\prime$ of $\lone$ such that $\Hull(I^\prime)=S$.
\end{enumerate}
\item For each well behaved closed ideal of $\lone$, $\Hull (I)$ is the unique $\homeo$-invariant closed subset $S^\prime$ of $\topspace$ such that $\Kernel(S^\prime)=I$;
\item If $I$ is a closed ideal of $\lone$, then $\Kernel\Hull(I)$ is the smallest well behaved closed ideal of $\lone$ containing $I$, and also the largest closed ideal $I^\prime$ of $\lone$ such that $\Hull(I^\prime)=\Hull(I)$.
\end{enumerate}
\end{theorem}

For spectral synthesis to hold in this model one needs to have mutually inverse bijections $\Hull$ and $\Kernel$ between $\Hulldomain$ and $\Kerneldomain$. Theorem~\ref{t:hull_kernel_setup_function_model} shows that it is only the injectivity of $\Hull$ on $\Hulldomain$ that is not automatic, because the inclusion $\Hulldomain_{wb}\subset\Hulldomain$ can be proper. For example, part (9) of Lemma~\ref{l:basic_Hull_properties} shows that (in fact also for non-closed ideals) $\Hull(I)=\emptyset$ precisely when $\Eone(I)$ is dense in $\coeffalg$. In particular, this will be the case for each badly behaved closed ideal $I$, and the next example shows that this non-injectivity on $\Hulldomain$ caused by the existence of proper badly behaved closed ideals can be rather substantial.

\begin{example}
If $\topspace$ consists of one point $x$, so that $\lone=\loneZ$, then each non-zero closed ideal $I$ of $\lone$ is badly behaved, as observed in Example~\ref{e:lone_example_behaviour}, hence $\Hull(I)=\emptyset$ for all such $I$.

The bijections in part (1) of Theorem~\ref{t:hull_kernel_setup_function_model} reduce to the trivial decreasing bijections between the tiny part ${\Hulldomain}_{wb}=\{ \{0\},\lone\}$ of $\Hulldomain$ and $\Kerneldomain=\{ \emptyset, \{x\} \}$.
\end{example}

Nevertheless, also in the presence of proper badly behaved closed ideals or of plain closed ideals, the bijection between $\Hulldomain_{wb}$ and $\Kerneldomain$ is informative, provided that it does not reduce to a triviality as for $\loneZ$. The next result describes when this degenerate situation occurs. It should be compared with \cite[Theorem~4.2]{DST}, stating that $\lone$ has only trivial closed ideals (or: only trivial self-adjoint closed ideals), precisely when $\dynsysshort$ is minimal and $\topspace$ has an infinite number of points.

\begin{corollary}\label{c:bijection_part_non_trivial}
The following are equivalent:
\begin{enumerate}
\item The only well behaved closed ideals of $\lone$ are $\{0\}$ and $\lone$;
\item $\dynsysshort$ is minimal, i.e., every point in $\topspace$ has dense orbit.
\end{enumerate}
\end{corollary}

\begin{proof}
Assume that (1) holds. For each $x\in \topspace$, $\Laurentseries{\kernel(\orbitclosure{x})}$ is a well behaved closed ideal, as a consequence of Corollary~\ref{c:characterisation_well_behaved_closed_ideals}. Since it is clearly proper, it equals $\{0\}$, and Corollary~\ref{c:characterisation_well_behaved_closed_ideals} implies that $\kernel(\orbitclosure{x})=\{0\}$, yielding that $\orbitclosure{x}=\topspace$. Conversely, if (2) holds, let $I$ be a proper well behaved closed ideal. Then, again by Corollary~\ref{c:characterisation_well_behaved_closed_ideals}, $I=\Laurentseries{\Eone(I)}$. Since $\Eone(I)$ is then a proper closed ideal of $\coeffalg$, $\hull(\Eone(I))$ is a non-empty $\homeo$-invariant subset of $\topspace$, hence equal to $\topspace$. Therefore $\Eone(I)=\kernel\hull(\Eone(I))=\{0\}$, hence $I=0$.
\end{proof}

Theorem~\ref{t:hull_kernel_setup_function_model} gives a hint as to when $\Hull$ could be injective on $\Hulldomain$: then all closed ideals must be well behaved, and in particular they will then all be self-adjoint. As is known \cite[Theorem~4.4]{DST} this can only occur if  $\dynsysshort$ is free. We will now proceed to show that freeness of $\dynsysshort$ is also sufficient for, hence equivalent with, the injectivity of $\Hull$ on $\Hulldomain$ and hence with spectral synthesis holding in this model. The following technical lemma is instrumental for this: its first two parts follow easily from \cite[Proposition~2.4]{DST} and the third and fourth part are trivial.

\begin{lemma}\label{l:killing_coefficients_elementwise_lemma} Suppose that $\dynsysshort$ is free. If $x\in\topspace$ and $N\geq 1$ are given, then there exist an open neighbourhood $U$ of $x$ and unimodular functions $\theta_1,\ldots,\theta_{4^N}\in \coeffalg$ with the following property: If $a=\sum_{n\in\Z}a_n\delta^n \in\lone$ is arbitrary, and $\frac{1}{4^N}\sum_{l=1}^{4^N}\theta_l a{\bar{\theta}}_l=\sum_{n\in\Z} a^\prime_n \delta^n$, then
\begin{enumerate}
\item $a^\prime_0=a_0$;
\item $a^\prime_n\rest{U}=0$, for $0< |n|\leq N$;
\item if $a_n(x)=0$, for some $n\in\Z$ and $x\in\topspace$, then $a_n^\prime(x)=0$;
\item $\Vert a^\prime_n\Vert\leq \Vert a_n\Vert$, for all $n\in\Z$.
\end{enumerate}
\end{lemma}

The previous result shows that, while staying in the same $\coeffalg$-subbimodule, one can locally annihilate any finite given set (not containing $a_0)$ of coefficients of $a$, while retaining $\Eone(a)$. Since zeroes of coefficients are preserved, and the norm of the coefficients does not increase, repeating this process a finite number of steps takes us to the global level.

\begin{lemma}\label{l:killing_coefficients_globally}
Suppose that $\dynsysshort$ is free. If $N\geq 1$ is given, then there exist finitely many unimodular functions $\theta_1,\ldots,\theta_{M}\in \coeffalg$ with the following property: If $a=\sum_{n\in\Z}a_n\delta^n \in\lone$ is arbitrary, and $\frac{1}{M}\sum_{k=1}^{M}\theta_k a{\bar{\theta}}_k=\sum_{n\in\Z} a^\prime_n \delta^n$, then
\begin{enumerate}
\item $a^\prime_0=a_0$;
\item $a^\prime_n=0$, for $0< |n|\leq N$;
\item if $a_n(x)=0$, for some $n\in\Z$ and $x\in\topspace$, then $a_n^\prime(x)=0$;
\item $\Vert a^\prime_n\Vert\leq \Vert a_n\Vert$, for all $n\in\Z$.
\end{enumerate}
\end{lemma}

\begin{proof} With $N$ given, we apply Lemma~\ref{l:killing_coefficients_elementwise_lemma} to each $x\in\topspace$, and obtain a neighbourhood $U_x$ of $x$ and unimodular functions $\theta_{x,1},\ldots,\theta_{x,4^N}\in \coeffalg$, such that the operator $L_x:\lone\to\lone$, mapping $a=\sum_{n\in\Z}a_n\delta^n \in\lone$ to $\frac{1}{4^N}\sum_{l=1}^{4^N}\theta_{x,l} a{\bar{\theta}}_{x,l}$, has the properties (1)--(4) as stated in Lemma~\ref{l:killing_coefficients_elementwise_lemma}. By compactness of $\topspace$, there are finitely many $x_1,\ldots,x_C$ such that $\topspace=\bigcup_{i=1}^C U_{x_i}$. The operator $L_{x_1}\circ\ldots\circ L_{x_C}$ is then described by a summation involving $M=4^{NC}$ unimodular functions as stated in the present Lemma. Moreover, the invariance parts (3) and (4) of Lemma~\ref{l:killing_coefficients_elementwise_lemma}, together with the annihilation of the coefficient with index in $\{-N,\ldots,-1,1,\ldots,N\}$ on $U_{x_i}$ once $L_{x_i}$ is applied, imply part (2), (3) and (4) of the present Lemma.
\end{proof}

\begin{corollary}\label{c:free_implies_all_closed_ideals_well_behaved} Suppose that $\dynsysshort$ is free.
 \begin{enumerate}
 \item If $a\in\lone$, then $\Eone(a)\in\overline{\coeffalg\cdot a\cdot\coeffalg}$.
 \item Every closed ideal of $\lone$ is well behaved.
\end{enumerate}
\end{corollary}

\begin{proof}
The second part follows trivially from the first. As to that, let $a=\loneelement{n}\in\lone$ and $\epsilon>0$. Choose $N\geq 1$ such that $\sum_{|n|>N}\Vert a_n\Vert<\epsilon$ and next apply Lemma~\ref{l:killing_coefficients_globally} to find an element $a^\prime=\sum_{n\in\Z}a^\prime_n\delta^n\in \coeffalg\cdot a\cdot\coeffalg$ such that
\begin{enumerate}
\item $a^\prime_0=a_0$;
\item $a^\prime_n=0$, for $0< |n|\leq N$;
\item $\Vert a^\prime_n\Vert\leq \Vert a_n\Vert$, for all $n\in\Z$.
\end{enumerate}
Then $\Vert\Eone(a)-a^\prime\Vert=\Vert a^\prime_0-a^\prime\Vert=\sum_{|n|>N}\Vert a^\prime_n\Vert\leq \sum_{|n|>N}\Vert a_n\Vert<\epsilon$.
\end{proof}

The main theorem on spectral synthesis holding in this model is now simply a matter of putting the pieces together. Needless to say, if the equivalent statements below hold, then all parts of Theorem~\ref{t:ideal_well_behaved} apply to all (then automatically well behaved) closed ideals of $\lone$.

\begin{theorem}\label{t:all_ideals_well_behaved}
The following are equivalent:
\begin{enumerate}
\item The maps $I \rightarrow \Hull(I)$ and $S \rightarrow \Kernel(S)$ are mutually inverse bijections between the set of closed ideals of $\lone$ and the set of $\homeo$-invariant closed subsets of $\topspace$;
\item Every closed ideal of $\lone$ is well behaved;
\item Every closed ideal of $\lone$ is self-adjoint;
\item Every primitive ideal of $\lone$ is well behaved;
\item Every closed ideal of $\lone$ is the intersection of well behaved primitive ideals;
\item Every closed ideal of $\lone$ is the intersection of primitive ideals;
\item $\dynsysshort$ is free.
\end{enumerate}
In that case, if $I$ is a closed ideal of $\lone$, and $S\subset\topspace$ is such that
\[
\hull(\Eone(I))=\overline{\bigcup_{x\in S}\orbit{x}},
\]
giving the $\homeo$-invariant closed subset $\hull(\Eone(I))$ as the closure of a union of orbits, then
\[
I=\bigcap_{x\in S}\piaper{x}
\]
establishes $I$ explicitly as an intersection as in part \ulb \textup{5}\urb.
\end{theorem}

\begin{proof}
The equivalence of (1) and (2) is immediate from the second part of Theorem~\ref{t:hull_kernel_setup_function_model}. It was already observed in Corollary~\ref{c:characterisation_well_behaved_closed_ideals} that, for general $\dynsysshort$, all well behaved closed ideals are self-adjoint; hence (2) implies (3). By \cite[Theorem~4.4]{DST}, (3) and (7) are equivalent. Corollary~\ref{c:free_implies_all_closed_ideals_well_behaved} shows that (7) implies (2). It is trivial that (2) implies (4). For any $x\in\perpoints$ and $\lambda\in\T$, the second part Proposition~\ref{p:behaviour_of_families} furnishes a primitive ideal $\piper{x}{\lambda}$ that is not well behaved, hence (4) implies (7). Thus the equivalence of (1), (2), (3), (4) and (7) has been established, and we turn to (5). Since arbitrary intersections of well behaved ideals are well behaved, according to Lemma~\ref{l:basic_E_properties}, (5) implies (2). If (2) holds, then (7) holds as well, and therefore the set $\setofperpoints$ of periodic points in Theorem~\ref{t:ideal_well_behaved} is necessarily empty, so that \eqref{e:explicit_intersection_of_orbitclosure_ideals} establishes every closed ideal as an intersection of well behaved primitive ideals. Hence (2) implies (5). Since obviously (5) implies (6) and (6) implies (3), the proof of the equivalence of (1) through (7) is now complete. The remaining statement is immediate from Theorem~\ref{t:ideal_well_behaved}.
\end{proof}

\begin{remark}\label{r:intersection_discussion}
Regardless of the dynamics, it is always true that every closed ideal of $\cstar$ is the intersection of a  number of the $\cstar$-counterparts of the primitive ideals $\piaper{x}\,(x\in\aperpoints)$ and $\piper{x}{\lambda}\,(x\in\perpoints,\lambda\in\T)$ \cite[Proposition~2]{T2}. Certainly every closed ideal of $\cstar$ is an intersection of primitive ideals, but that this standard family is always sufficient is remarkable. If $\topspace$ is metrizable, then it is known (cf.\ \cite{GR}, where a generalisation of the Effros-Hahn conjecture in \cite{EH} is proved) that the primitive ideals of $\cstar$ are precisely the ones in our standard family, and \cite[Proposition~2]{T2} is then obvious, but for non-metrizable $\topspace$ there seems to be no a priori guarantee for this to hold.

One might hope that each closed ideal of $\lone$ is the intersection of primitive ideals. However, such an intersection will always be self-adjoint, and hence a necessary condition for this is that $\dynsysshort$ should be free. As Theorem~\ref{t:all_ideals_well_behaved} shows, this condition is also sufficient, and in that case the primitive ideals $\piaper{x}\,(x\in\aperpoints)$ of $\lone$ are already sufficiently many. If $\topspace$ is metrizable, then \cite{GR} implies again that this is the complete set of primitive ideals of $\lone$, but for non-metrizable $\topspace$ this may no longer be true.
\end{remark}

\section{Spectral synthesis: $\loneZ$-model}\label{sec:l_one_model}

In this section we study spectral synthesis in a model analogous to the Fourier transform for $\loneZ$. Compared with the previous section, the roles of the operators $\Hull$ and $\Kernel$ are now taken over by the operators $\Zeroes$ and $\Ideal$, respectively, which are the analogues of the usual hull and kernel operators, respectively, for $L^1(G)$.

To start with, we define the injective contraction $\Fourier:\lone\to C(\XtimesT)$ as
\[
\Fourier(a)(x,\lambda)=\sum_n \lambda^n a_n(x)\quad(a=\loneelement{n}\in\lone,x\in\topspace,\lambda\in\T).
\]
It collects the Fourier transforms of all maps $n\to a_n(x)$, for $x\in\topspace$, and when $\topspace$ consists of one point it is the usual Fourier transform. The following properties are routine to verify.

\begin{lemma}\label{l:basic_Fourier_properties}Let $a=\loneelement{n}\in\lone$, $(x,\lambda)\in\XtimesT$, $k\in\Z$, and $f\in\coeffalg$. Then:
\begin{enumerate}
\item $\Fourier(1)=1$;
\item $\Fourier(a\cdot\delta^k)(x,\lambda)=\lambda^k\Fourier(a)(x,\lambda)$;
\item $\Fourier(\delta^k\cdot a)(x,\lambda)=\lambda^k\Fourier(a)(\homeo^{-k}x,\lambda)$;
\item $\Fourier(f\cdot a)(x,\lambda)=f(x)\Fourier(a)(x,\lambda)$;
\item $\Fourier(a\cdot f)(x,\lambda)=\sum_n \lambda^n a_n(x) f(\homeo^{-n}x)$;
\item $\Fourier(a^\ast)(x,\lambda)=\overline{\sum_n \lambda^n a_n(\homeo^n x)}$.
\end{enumerate}
\end{lemma}

When $L$ is a linear subspace of $\lone$, let
\[
\Zeroes(L)=\{(x,\lambda)\in\XtimesT : \Fourier(a)(x,\lambda)=0\textup{ for all }a\in L\}
\]
be the possibly empty set of common zeroes of all transforms of elements of $L$. If $\topspace$ consists of one point, and $I$ is an ideal of $\loneZ$, then $\Zeroes(I)$ is the usual hull of $I$.

The following is readily established, using Lemma~\ref{l:basic_Fourier_properties} for the final statement.

\begin{lemma}\label{l:basic_Zeroes_properties} Let $L$ be a linear subspace of $\lone$. Then:
\begin{enumerate}
\item $\Zeroes(L)$ is a closed subset of $\XtimesT$;
\item $\Zeroes(L)=\XtimesT$ if and only if $L=\{0\}$, and $Z(L)=\emptyset$ if $L=\lone$.
\item If $L^\prime$ is a linear subspace of $\lone$, and $L^\prime\subset L$, then $\Zeroes(L^\prime)\supset\Zeroes(L)$;
\item $\Zeroes(\bar L)=\Zeroes(L)$;
\item $\Zeroes(L)\supset\Hull(L)\times \T$;
\item If $L\subset\coeffalg$, then $\Zeroes(I)=\hull(I)\times \T$.
\end{enumerate}
If $\{L_\alpha : \alpha\in A\}$ is a collection of linear subspaces of $\lone$, then:
\begin{enumerate}
\item[(7)] $\Zeroes(\sum_{\alpha\in A} L_\alpha)=\bigcap_{\alpha\in A}\Zeroes(L_\alpha)$;
\item[(8)] $\Zeroes(\bigcap_{\alpha\in A} L_\alpha)\supset\bigcup_{\alpha\in A}\Zeroes(L_\alpha)$.
\end{enumerate}
If $I$ is an ideal of $\lone$, then $\Zeroes(I)$ is a $\prodhomeo$-invariant closed subset of $\XtimesT$.
\end{lemma}

\begin{remark}\label{r:Z_I_ast_Z_I}\quad
\begin{enumerate}
\item Note that it is not asserted, not even for a closed ideal $I$, that $\Zeroes(I)=\emptyset$ if \emph{and only if} $I=\lone$.  The question whether $\Zeroes(I)\neq\emptyset$ for every proper closed ideal of $\lone$ touches upon one of the basic issues concerning the relation between $\lone$ and $\cstar$; see Proposition~\ref{p:C_star_closure_proper}.
\item It is not obvious at this point that $\Zeroes(I)=\Zeroes(I^\ast)$, for each closed ideal $I$ of $\lone$. This will be established later in Corollary~\ref{c:Zeroes_I_ast_Zeroes_I}.
\end{enumerate}
\end{remark}

The candidate inverse $\Ideal$ for $\Zeroes$ is defined in two steps. If $S\subset\XtimesT$, define
\[
\Idealtilde(S)=\{a\in\lone : \Fourier(a)(x,\lambda)=0\textup{ for all }(x,\lambda)\in S\},
\]
and
\[
\Ideal(S)=\{a\in\lone : a\cdot f\in\Idealtilde(S)\textup{ for all }f\in\coeffalg\}.
\]
If $X$ consists of one point, then $\Idealtilde(S)$ and $\Ideal(S)$ coincide and are the usual kernel of $S\subset\T$.

We have the following elementary properties, which are routine to verify, using the injectivity of the Fourier transform on $\loneZ$ for part (6).

\begin{lemma}\label{l:basic_Ideal_properties}Let $S\subset\XtimesT$. Then:
\begin{enumerate}
\item $\Idealtilde(S)$ is a closed linear subspace of $\lone$;

    \item $\Idealtilde(S)=\lone$ if and only if $S=\emptyset$;
    \item If $\bar S=\XtimesT$, then $\Idealtilde(S)=\{0\}$;
\item If $S^\prime\subset S$, then $\Idealtilde(S^\prime)\supset\Idealtilde(S)$;
\item $\Idealtilde(\bar S)=\Idealtilde(S)$;
\item $\Idealtilde(A\times\T)=\Laurentseries{\kernel(A)}$, for all $A\subset\topspace$;
\item $\Ideal(S)\subset\Idealtilde(S)$;

\item $\Ideal(S)=\lone$ if and only if $S=\emptyset$;
\item If $\bar S=\XtimesT$, then $\Ideal(S)=\{0\}$;
\item If $S^\prime\subset S$, then $\Ideal(S^\prime)\supset\Ideal(S)$;
\item $\Ideal(\bar S)=\Ideal(S)$;
\item $\Ideal(A\times\T)\subset\Laurentseries{\kernel(A)}$, for all $A\subset\topspace$.
\end{enumerate}
If $S\subset\XtimesT$ is $\prodhomeo$-invariant, then:
\begin{enumerate}
\item[(13)] $\Idealtilde(S)$ is a closed subspace of $\lone$, which is invariant under the left action of $\coeffalg$, and under left and right multiplication with $\delta^k$, for $k\in\Z$.
\item[(14)] $\Ideal(S)$ is a closed ideal of $\lone$;
\item[(15)] $\Ideal(S)=\Idealtilde(S)$ if and only if $\Idealtilde(S)$ is a closed ideal of $\lone$.
\end{enumerate}
If $\{S_\alpha : \alpha\in A\}$ is a collection of subsets of $\XtimesT$, then:
\begin{enumerate}
\item[(16)] $\Idealtilde(\bigcup_{\alpha\in A}S_\alpha)=\bigcap_{\alpha\in A}\Idealtilde(S_\alpha)$;
\item[(17)] $\Idealtilde(\bigcap_{\alpha\in A}S_\alpha)\supset \sum_{\alpha\in A}\Idealtilde(S_\alpha)$;
\item[(18)] $\Ideal(\bigcup_{\alpha\in A}S_\alpha)=\bigcap_{\alpha\in A}\Ideal(S_\alpha)$;
\item[(19)] $\Ideal(\bigcap_{\alpha\in A}S_\alpha)\supset \sum_{\alpha\in A}\Ideal(S_\alpha)$.
\end{enumerate}
\end{lemma}

\begin{lemma}\label{l:Zeroes_Ideal_composition}\quad
\begin{enumerate}
\item If $L$ is a linear subspace of $\lone$, then $\Idealtilde\Zeroes(L)\supset L$.
\item If $S\subset\XtimesT$, then $\Zeroes\Ideal(S)\supset\Zeroes\Idealtilde(S)\supset S$;
\item If $I$ is an ideal of $\lone$, then $\Ideal\Zeroes(I)\supset I$;
\end{enumerate}
\end{lemma}

\begin{proof}
The first part is obvious, and so is the inclusion $\Zeroes\Idealtilde(S)\supset S$. Since $\Ideal(S)\subset\Idealtilde(S)$, part (2) follows. As to the third part, we know that $I\subset\Idealtilde(\Zeroes(I))$. Since $I$ is an ideal, $I\cdot\coeffalg=I\subset \Idealtilde(\Zeroes(I))$, so that $I\subset\Ideal(\Zeroes(I))$.
\end{proof}

Part (3) of Lemma~\ref{l:basic_Zeroes_properties}, part (10) of Lemma~\ref{l:basic_Ideal_properties}, and part (2) and (3) of Lemma~\ref{l:Zeroes_Ideal_composition} together imply that we are in the context of Appendix~\ref{sec:appendix}, when we let $\Zeroes$ assign a $\prodhomeo$-invariant subset of $\XtimesT$ to an ideal of $\lone$, with $\Ideal$ going in the opposite direction. Consequently,
\[
\Zeroes\Ideal\Zeroes(I)=\Zeroes(I),
\]
for each ideal $I$ of $\lone$, and
\[
\Ideal\Zeroes\Ideal(S)=\Ideal(S),
\]
for each $\prodhomeo$-invariant subset $S$ of $\XtimesT$. Of course, one can restrict the domain of $\Zeroes$ to the closed ideals of $\lone$, and that of $\Ideal$ to the closed $\prodhomeo$-invariant subsets of $\XtimesT$; that is also a context for Appendix~\ref{sec:appendix}.

Before invoking the results of Appendix~\ref{sec:appendix} we need to collect more material, such as a description of the fixed points of $\Ideal\Zeroes$ in the set of closed ideals of $\lone$. The first step is to consider well behaved ideals.

\begin{proposition}\label{p:zeroes_of_well_behaved_ideal}
If $I$ is a well behaved ideal of $\lone$, then $\Zeroes(I)=\Hull(I)\times\T$. The map $\Zeroes$ and $\Ideal$ are mutually inverse bijections between the set of all well behaved closed ideals of $\lone$ on the one hand, and the subset $\{A\times\T : A\subset\topspace \textup{ closed and $\homeo$-invariant}\}$ of the set of all $\prodhomeo$-invariant closed subsets of $\XtimesT$ on the other hand.
\end{proposition}

\begin{proof}
Lemma~\ref{l:basic_Zeroes_properties} shows that $\Zeroes(I)\supset\Hull(I)\times\T$. On the other hand, since $\Eone(I)\subset I$, it is immediate that $\Zeroes(I)\subset\Zeroes(\Eone(I))=\hull(\Eone(I))\times\T=\Hull(I)\times\T$. Hence $\Zeroes(I)=\Hull(I)\times\T$.
Since $\Hull$ is injective on the set of well behaved closed ideals of $\lone$, so is $\Zeroes$. For surjectivity, assume that $A\subset\topspace$ is closed and $\homeo$-invariant. Lemma~\ref{l:basic_Ideal_properties} shows that $\Idealtilde(A\times\T)=\Laurentseries{\kernel(A)}$, but, since $A$ is $\homeo$-invariant, this is already an ideal, so $\Ideal(A\times\T)=\Idealtilde(A\times\T)=\Laurentseries{\kernel(A)}$ is a well behaved closed ideal, and clearly $\Hull(\Ideal(A\times\T))=\Hull(\Laurentseries{\kernel(A)})=\hull\kernel(A)=A$. From what we have already seen, we conclude that $\Zeroes(\Ideal(A\times T))=[\Hull(\Ideal(A\times\T))]\times\T=A\times\T$. Hence $\Zeroes$ is surjective, and $\Zeroes$ and $\Ideal$ are mutually inverse bijections between these restricted domains.
\end{proof}

Hence the well behaved ideals are fixed under $\Ideal\Zeroes$ but, quite contrary to Section~\ref{sec:function_space_model}, there are others. In order to obtain a full description of these fixed points, we investigate our three standard families of ideals.

We start with the well behaved (self-adjoint) closed ideals $\piaper{x}$, for $x\in\aperpoints$. The following is immediate from Proposition~\ref{p:zeroes_of_well_behaved_ideal} and (for part (3))  Lemma~\ref{l:order_inflection}.

\begin{corollary}\label{c:I_and_Z_for_P_x} Let $x\in\aperpoints$. Then:
\begin{enumerate}
\item $\Zeroes(\piaper{x})=\left[\,\orbitclosure{x}\,\right]\times\T$;
    \item $\Ideal(\left[\,\orbitclosure{x}\,\right]\times\T)=\piaper{x}$;
\item If $I$ is an ideal of $\lone$, then $I\subset\piaper{x}$ if and only if $\Zeroes(I)\supset\Zeroes(\piaper{x})$.
\end{enumerate}
\end{corollary}

Actually, part (3) can be improved quite a bit, which will be instrumental in the proof of the key Proposition~\ref{p:key_pre_image_proposition}.

\begin{proposition}\label{p:included_in_P_x} Let $x\in\aperpoints$, and let $I$ be an ideal of $\lone$. Then the following are equivalent:
\begin{enumerate}
\item $\Zeroes(I)\cap[\,\orbit{x}\,]\times \T\neq\emptyset$;
\item $\Zeroes(I)\cap \{x\}\times\T\neq\emptyset$;
\item $\Zeroes(I)\supset \{x\}\times\T$;
\item $\Zeroes(I)\supset\left[\,\orbitclosure{x}\,\right]\times\T$;
\item $I\subset\piaper{x}$.
\end{enumerate}
\end{proposition}

\begin{proof}
Obviously (5) implies (4), (4) implies (3), and (3) implies (2); (2) and (1) are equivalent since $\Zeroes(I)$ is $\prodhomeo$-invariant. We will prove that (2) implies (5); so assume that $(x,\lambda)\in\Zeroes(I)$ for some $\lambda\in T$. Let $a=\loneelement{n}\in I$. We will start by showing that $a_0(x)=0$. Let $\epsilon>0$, and choose $N\geq 1$ such that $\sum_{|n|>N}\Vert a_n\Vert<\epsilon$. Since $x\in\aperpoints$, there exists $f\in\coeffalg$ with $\Vert f\Vert=1$ and such that $f(\homeo^n x)=0$, for $0<|n|\leq N$, while $f(x)=1$. Now $a\cdot f=\sum_n a_n\cdot(f\circ\homeo^{-n})\delta^n$, and since $a\cdot f$ is in $I$ we have
\begin{align*}
0&=\Fourier(a\cdot f)(x,\lambda)\\
&=\sum_n \lambda^n a_n(x) f(\homeo^{-n}x)\\
&=a_0(x) + \sum_{|n|>N}\lambda^n a_n(x)f(\homeo^{-n}x).
\end{align*}
Since the latter term is at most $\epsilon$ in absolute value, $|a_0(x)|\leq\epsilon$. Hence $a_0(x)=0$.

Since we know $\Zeroes(I)$ to be a closed $\prodhomeo$-invariant subset of $\XtimesT$, $(x^\prime,\lambda)$ is likewise in $\Zeroes(I)$, for all $x^\prime\in\orbitclosure{x}$. Hence the above argument shows that $a_0\rest{\orbitclosure{x}}=0$. Since $a\cdot\delta^k$ is in $I$, for all $k\in\Z$, we can now conclude that $a_n\rest{\orbitclosure{x}}=0$, for all $n\in\Z$. Hence, by part (1) of Proposition~\ref{p:belonging_to_kernel}, $a$ is in $\piaper{x}$.
\end{proof}

We now turn to the badly behaved self-adjoint closed ideals $\piper{x}{\lambda}$, for $x\in\perpoints$ and $\lambda\in\T$. As we will see in Corollary~\ref{c:I_and_Z_for_P_x_lambda}, they are quite well behaved as far as $\Ideal$ and $\Zeroes$ are concerned.

\begin{proposition}\label{p:included_in_P_x_lambda} Let $x\in\perpnew$, $\lambda\in\T$, and let $I$ be an ideal of $\lone$. Then the following are equivalent:
\begin{enumerate}
\item $\Zeroes(I)\cap \{x\}\times\standardsetofroots\neq\emptyset$;
\item $\Zeroes(I)\supset \{x\}\times\standardsetofroots$;
\item $\Zeroes(I)\supset \left[\,\orbit{x}\,\right]\times\standardsetofroots$;
\item $I\subset\piper{x}{\lambda}$.
\end{enumerate}
\end{proposition}

\begin{proof}
Assume that (4) holds, and that $\mu^p=\lambda$. Suppose $a=\loneelement{n}\in I\subset\piper{x}{\lambda}$. Then \eqref{e:vanishing} shows that
\[
\sum_{l\in\Z}\mu^{lp} a_{lp + j}(x^\prime) = 0,
\]
for all $j\in\{0,1,\ldots,p-1\}$, and all $x^\prime\in\orbit{x}$. Multiplying this relation by $\mu^j$, and summing the result over the set of all $j$, shows that $\Fourier(a)(x^\prime,\mu)=0$. Hence (4) implies (3). Certainly (3) implies (2), and (2) implies (1). We will show that (1) implies (4). Suppose, then, that $\mu^p=\lambda$ and that $(x,\mu)\in\Zeroes(I)$. Fix $a\in I$. For all $f\in\coeffalg$, $a\cdot f$ is in $I$, hence as in the proof of Proposition~\ref{p:included_in_P_x} we know that
\begin{align*}
0&=\Fourier(a\cdot f)(x,\mu)\\
&=\sum_n \mu^n a_n(x) f(\homeo^{-n}x)\\
&=\sum_{j=0}^{p-1} \left[\sum_{l\in\Z} \mu^{lp+j}a_{lp+j}(x)f(\homeo^{-j}x)\right].
\end{align*}
For $j_0\in\{0,1,\ldots,p-1\}$ fixed, choose $f$ such that $f(\homeo^{-j}x)=0$ for $j_0\neq j\in\{0,1,\ldots,p-1\}$, and $f(\homeo^{-j_0} x)=1$. Then in the above equation only the inner summation for $j=j_0$ survives, and together with $\mu^p=\lambda$ this yields
\[
\sum_{l\in\Z} \lambda^l a_{lp+j_0}(\homeo^{-j_0}x)=0.
\]
The second part of Proposition~\ref{p:belonging_to_kernel} then shows that $a\in\piper{x}{\lambda}$.
\end{proof}

\begin{corollary}\label{c:I_and_Z_for_P_x_lambda} Let $x\in\perpnew$, and $\lambda\in\T$. Then:
\begin{enumerate}
\item $\Zeroes(\piper{x}{\lambda})=\left[\,\orbit{x}\,\right]\times\standardsetofroots$;
\item $\Ideal(\left[\,\orbit{x}\,\right]\times\standardsetofroots)=\piper{x}{\lambda}$;
\item If $I$ is an ideal of $\lone$, then the following are equivalent:
\begin{enumerate}
\item $\Zeroes(I)\cap\Zeroes(\piper{x}{\lambda})\neq\emptyset$;
\item $\Zeroes(I)\supset\Zeroes(\piper{x}{\lambda})$;
\item $I\subset\piper{x}{\lambda}$.
\end{enumerate}
\end{enumerate}
\end{corollary}

\begin{proof}
Proposition~\ref{p:included_in_P_x_lambda} shows that $\Zeroes(\piper{x}{\lambda})\supset\left[\,\orbit{x}\,\right]\times\standardsetofroots$. For the reverse inclusion, assume that $(x^\prime,\mu)\in\Zeroes(\piper{x}{\lambda})$, for some $x^\prime\in\topspace$ and $\mu\in\T$. Since $1-(1/\lambda)\delta^p$ is in $\piper{x}{\lambda}$, we see that $1-(\mu^p/\lambda)=0$, hence $\mu^p=\lambda$. In order to show that we must have $x^\prime\in\orbit{x}$, note that $\piper{x}{\lambda}\supset\piperintersection{x}$, hence $\Zeroes(\piper{x}{\lambda})\subset\Zeroes(\piperintersection{x})$, implying that $(x^\prime,\mu)\in\Zeroes(\piperintersection{x})$.
Since $\piperintersection{x}$ is a well behaved ideal, Lemma~\ref{l:basic_Zeroes_properties} shows that $\Zeroes(\piperintersection{x})=\Hull(\piperintersection{x})\times\T=\left[\,\orbit{x}\,\right]\times\T$. Hence $x^\prime\in\orbit{x}$. This concludes the proof of part (1). For part (2), we note that certainly $\Ideal\Zeroes(\piper{x}{\lambda})\supset\piper{x}{\lambda}$. On the other hand, $\Zeroes(\Ideal\Zeroes(\piper{x}{\lambda}))=\Zeroes(\piper{x}{\lambda})=\left[\,\orbit{x}\,\right]\times\standardsetofroots$ by part (1), hence Proposition~\ref{p:included_in_P_x_lambda} shows that $\Ideal\Zeroes(\piper{x}{\lambda})\subset\piper{x}{\lambda}$. Therefore, $\Ideal\Zeroes(\piper{x}{\lambda})=\piper{x}{\lambda}$. Then (2) follows from this and an application of $\Ideal$ to the equality in (1).
Part (3) follows easily from part (1), the $\prodhomeo$-invariance of $\Zeroes(I)$, and Proposition~\ref{p:included_in_P_x_lambda}.
\end{proof}

Finally, for our third family, Proposition~\ref{p:zeroes_of_well_behaved_ideal} and Lemma~\ref{l:order_inflection} imply the following.

\begin{corollary}\label{c:I_and_Z_for_Q_x}
Let $x\in\perpoints$. Then:
\begin{enumerate}
\item $\Zeroes(\piperintersection{x})=\left[\,\orbit{x}\,\right]\times\T$;
\item $\Ideal(\left[\,\orbit{x}\,\right]\times\T)=\piperintersection{x}$;
\item If $I$ is an ideal of $\lone$, then $I\subset\piperintersection{x}$ if and only if $\Zeroes(I)\supset\Zeroes(\piperintersection{x})$.
\end{enumerate}
\end{corollary}

Before proceeding, let us collect a few consequences of the results thus far.

\begin{proposition}\label{p:C_star_closure_proper}
Let $I$ be an ideal of $\lone$. Then the closure of $I$ in $\cstar$ is a proper closed ideal of $\cstar$ if and only if $\Zeroes(I)\neq\emptyset$.
\end{proposition}

\begin{proof}
If $\Zeroes(I)\neq\emptyset$, then Proposition~\ref{p:included_in_P_x} and Proposition~\ref{p:included_in_P_x_lambda} imply that $I$ is contained in an ideal $\piaper{x}$, for some $x\in\aperpoints$, or in an ideal $\piper{x}{\lambda}$, for some $x\in\perpoints$ and $\lambda\in\T$. Hence it is contained in the kernel of the extension of the involutive representation $\pi_x$ or $\pi_{x,\lambda}$ to $\cstar$. As these kernels are proper closed ideals of $\cstar$, the closure of $I$ in $\cstar$ is also proper. Conversely, if the closure of $I$ in $\cstar$ is proper, then by \cite[Proposition~2]{T2}, this closure is the intersection of a number of kernels of such extended involutive representations. Taking the intersection with $\lone$ then implies that $I$ is contained in an ideal $\piaper{x}$, for some $x\in\aperpoints$, or in an ideal $\piper{x}{\lambda}$, for some $x\in\perpoints$ and $\lambda\in\T$. Proposition~\ref{p:included_in_P_x} and Proposition~\ref{p:included_in_P_x_lambda} then show that $\Zeroes(I)\neq\emptyset$.
\end{proof}

\begin{remark}\label{r:C_closure_proper}If it were true that $\Zeroes(I)\neq\emptyset$, for all proper ideals of $\lone$, then many of the known results for $\cstar$ that relate the dynamics to the ideal structure of the algebra would have immediate counterparts for $\lone$. For example, it is known (cf.\ \cite[Theorem~4.2]{DST}) that $\lone$ has only trivial closed ideals precisely when $\topspace$ has an infinite number of points, and $\dynsysshort$ is minimal. The difficult part is to conclude the minimality from the dynamics, but if we could pass from proper closed ideals of $\lone$ to proper closed ideals of $\cstar$, then this would be obvious from its counterpart for $\cstar$ (cf.\ \cite[Theorem~5.3]{T1})
\end{remark}

\begin{proposition}
If $I$ is a badly behaved ideal of $\lone$, then $\Zeroes(I)\subset\perpoints\times\T$.
\end{proposition}

\begin{proof}
If $\Zeroes(I)\subsetneqq \perpoints\times\T$ then Proposition~\ref{p:included_in_P_x} implies that $I\subset\piaper{x}$, for some $x\in\aperpoints$. This contradicts that $\piaper{x}$ is not badly behaved.
\end{proof}

We will now use the fact that we are in the setup of Appendix~\ref{sec:appendix}, when we let $\Zeroesdomain$ be the closed ideals of $\lone$, and $\Idealdomain$ the $\prodhomeo$-invariant subsets of $\XtimesT$, with $\Zeroes$ mapping the former into the latter, and $\Ideal$ going in the opposite direction. It is then possible to describe the fixed points of $\Ideal\Zeroes$: according to Corollary~\ref{c:corollary_of_three_maps_lemma}, these are precisely the closed ideals of the form $\Ideal(S)$, with $S$ a $\prodhomeo$-invariant subset of $\XtimesT$. To make such ideals explicit, the following result is needed.

\begin{proposition}\label{p:key_pre_image_proposition}\quad
\begin{enumerate}
\item Let $x\in\aperpoints$, and suppose $\emptyset\neq S\subset[\,\orbit{x}\,]\times\T$ is $\prodhomeo$-invariant. Then $\Ideal(S)=\piaper{x}$.
\item Let $x\in\perpnew$, and suppose $\emptyset\neq S\subset\Zeroes(\piper{x}{\lambda})=[\,\orbit{x}\,]\times\standardsetofroots$ is $\prodhomeo$-invariant. Then $\Ideal(S)=\piper{x}{\lambda}$.
\end{enumerate}
\end{proposition}

\begin{proof}
As to (1), since $\Zeroes\Ideal(S)\supset S\neq\emptyset$, the condition in part (1) of Proposition~\ref{p:included_in_P_x} is satisfied for $\Ideal(S)$, and we conclude that $\Ideal(S)\subset\piaper{x}$. On the other hand, certainly $S\subset[\,\orbitclosure{x}\,]\times\T=\Zeroes(\piaper{x})$, hence $\Ideal(S)\supset\Ideal\Zeroes(\piaper{x})=\piaper{x}$. Thus $\Ideal(S)=\piaper{x}$. The second part is proved similarly, using Proposition~\ref{p:included_in_P_x_lambda}.
\end{proof}

\begin{theorem}\label{t:fixed_points_of_IZ}
Let $I$ be an ideal of $\lone$. Then the following are equivalent:
\begin{enumerate}
\item $\Ideal\Zeroes(I)=I$;
\item There exist \ulb possibly empty\urb\ sets $\setofaperpoints\subset\aperpoints$, $\setofperpoints\subset\perpoints$, and, for each $x\in\setofperpoints$, a set $\T_x\subset\T$, such that
\begin{equation}\label{e:I_Z_fixed_is_intersection}
I=\bigcap_{x\in\setofaperpoints}\piaper{x}\,\bigcap_{x\in\setofperpoints}\bigcap_{\lambda\in\T_x}\piper{x}{\lambda};
\end{equation}
\item $I$ is the kernel of an involutive representation of $\lone$;
\item If $I^\prime$ is an ideal of $\lone$, then $I^\prime\subset I$ if and only if $\Zeroes(I^\prime)\supset\Zeroes(I)$.
\end{enumerate}
In that case, $I$ is a self-adjoint closed ideal of $\lone$. If $\setofaperpoints\subset\aperpoints$, $\setofperpoints\subset\perpoints$, and, for each $x\in\setofaperpoints\cup\setofperpoints$, $\T_x\subset\T$, are such that
\[
\Zeroes(I)=\bigcup_{x\in\setofaperpoints}\bigcup_{\lambda\in\T_x}[\,\orbit{x}\,]\times\{\lambda\}\,\bigcup_{x\in\setofperpoints}\bigcup_{\mu\in\T_x}[\,\orbit{x}\,]\times\{\mu\},
\]
then
\[
I=\bigcap_{x\in\setofaperpoints}\piaper{x}\,\bigcap_{x\in\setofperpoints}\bigcap_{\mu\in\T_x}\piper{x}{\mu^p}
\]
is an explicit intersection as in part \ulb\textup{2}\urb.
\end{theorem}

\begin{proof}
From Appendix~\ref{sec:appendix} we know that the ideals as in part (1) are precisely the ideals of the form $\Ideal(S)$, for $S$ a $\prodhomeo$-invariant subset of $\XtimesT$. We will show that these ideals are precisely the intersections as in the right hand side of \eqref{e:I_Z_fixed_is_intersection}.

If $S\subset\XtimesT$ is $\prodhomeo$-invariant, then it is evidently possible to find subsets $\setofaperpoints\subset\aperpoints$, $\setofperpoints\subset\perpoints$, and, for each $x\in\setofaperpoints\cup\setofperpoints$ a set $\T_x\subset\T$ such that
\[
S=\bigcup_{x\in\setofaperpoints}\bigcup_{\lambda\in\T_x}[\,\orbit{x}\,]\times\{\lambda\}\,\bigcup_{x\in\setofperpoints}\bigcup_{\mu\in\T_x}[\,\orbit{x}\,]\times\{\mu\}.
\]
Now Proposition~\ref{p:key_pre_image_proposition} shows that $\Ideal([\,\orbit{x}\,]\times\{\lambda\})$ equals $\piaper{x}$, if $x\in\aperpoints$, and that it equals $\piper{x}{\mu^p}$, if $x\in\perpoints$. Hence the penultimate statement in Lemma~\ref{l:basic_Ideal_properties} shows that $\Ideal(S)$ is an intersection as in the right hand side of \eqref{e:I_Z_fixed_is_intersection}. Conversely, all ideals $I$ that can be written as an intersection in \eqref{e:I_Z_fixed_is_intersection} can be obtained as $\Ideal(S)$ for a suitable $\prodhomeo$-invariant $S\subset\XtimesT$: according to Corollary~\ref{c:I_and_Z_for_P_x}, Corollary~\ref{c:I_and_Z_for_P_x_lambda}, and the penultimate statement of Lemma~\ref{l:basic_Ideal_properties}, if $S=\bigcup_{x\in\setofaperpoints}\Zeroes(\piaper{x})\,\bigcup_{x\in\setofperpoints}\,\bigcup_{\lambda\in\T_x}\Zeroes(\piper{x}{\lambda})$, then $\Ideal(S)=I$. Thus (1) and (2) are equivalent.

If $I$ is an intersection as in \eqref{e:I_Z_fixed_is_intersection}, note that each of these ideals is the kernel of an involutive representation. Hence $I$ is the kernel of the Hilbert sum of these representations. Hence (2) implies (3). If (3) holds, then we need only extend the given involutive representation $\pi$ to an involutive representation $\tilde\pi$ of $\cstar$, use \cite[Proposition~2]{T2} to write $\Ker\tilde\pi$ as an intersection of the counterparts of the $\piaper{x}$ and $\piper{x}{\lambda}$ for $\cstar$, and take the intersection of the ensuing relation with $\lone$ to see that (3) implies (2). The equivalence of (1) and (4) is a restatement of Lemma~\ref{l:order_inflection} in the present context.

This completes the proof of the equivalences. Any such ideal is clearly self-adjoint, and the remaining statement has been established during the previous part of the proof.
\end{proof}

Now that the fixed points of $\Ideal\Zeroes$ have been identified, the results in Appendix~\ref{sec:appendix} yield the following.

\begin{theorem}\label{t:hull_kernel_setup_L_1_model}
Let $\Zeroesdomain$ be the set of all closed ideals of $\lone$ and let $\Idealdomain=\{\Zeroes(I) : I\in\Zeroesdomain\}$ be the ensuing subset of the set of all $\prodhomeo$-invariant closed subsets of $\XtimesT$, both ordered by inclusion. Then $\Zeroes: \Zeroesdomain\to \Idealdomain$ and $\Ideal: \Idealdomain\to \Zeroesdomain$ are decreasing, $\Ideal\circ\Zeroes (I)\succ I$ for all $I\in\Zeroesdomain$, and $\Zeroes\circ\Ideal=\idmap_{\Idealdomain}$. Let ${\Zeroesdomain}_{\textup{invrep}}$ be the set of all kernels of involutive representations of $\lone$. Then:
\begin{enumerate}
\item $\Zeroes:{\Zeroesdomain}_{\textup{invrep}}\to\Idealdomain$ and $\Ideal:\Idealdomain\to {\Zeroesdomain}_{\textup{invrep}}$ are mutually inverse bijections;
\item The following are equivalent:
\begin{enumerate}
\item $\Zeroes$ is injective on $\Zeroesdomain$;
\item Each closed ideal $I$ of $\lone$ is of the form $\Ideal(S)$ for some $S\in\Idealdomain$;
\item Each closed ideal of $\lone$ is the kernel of an involutive representation of $\lone$;
\item For each $S\in\Idealdomain$, $\Ideal(S)$ is the unique closed ideal $I^\prime$ of $\lone$ such that $\Zeroes(I^\prime)=S$.
\end{enumerate}
\item For each ideal $I$ that is the kernel of an involutive representation of $\lone$, $\Zeroes(I)$ is the unique element $S$ of $\Idealdomain$ such that $\Zeroes(S)=I$;
\item If $I$ is a closed ideal of $\lone$, then $\Ideal\Zeroes(I)$ is the smallest kernel of an involutive representation of $\lone$ that contains $I$, and it also the largest closed ideal $I^\prime$ of $\lone$ such that $\Zeroes(I^\prime)=\Zeroes(I)$. It is self-adjoint and,
if $\setofaperpoints\subset\aperpoints$, $\setofperpoints\subset\perpoints$, and, for each $x\in\setofaperpoints\cup\setofperpoints$, $\T_x\subset\T$, are such that
\[
\Zeroes(I)=\bigcup_{x\in\setofaperpoints}\bigcup_{\lambda\in\T_x}[\,\orbit{x}\,]\times\{\lambda\}\,\bigcup_{x\in\setofperpoints}\bigcup_{\mu\in\T_x}[\,\orbit{x}\,]\times\{\mu\},
\]
then
\[
\Ideal\Zeroes(I)=\bigcap_{x\in\setofaperpoints}\piaper{x}\,\bigcap_{x\in\setofperpoints}\bigcap_{\mu\in\T_x}\piper{x}{\mu^p}
\]
\end{enumerate}
\end{theorem}

\begin{proof}
Everything is clear from the results in Appendix~\ref{sec:appendix}, except for the explicit intersection in part (4). As to this, note that $\Ideal\Zeroes(\Ideal\Zeroes(I))=\Ideal\Zeroes(I)$, hence Theorem~\ref{t:fixed_points_of_IZ} shows how $\Ideal\Zeroes(I)$ can be written as an intersection corresponding to a decomposition of $\Zeroes(\Ideal\Zeroes(I))=\Zeroes(I)$.
\end{proof}

We can now resolve the second issue raised in Remark~\ref{r:Z_I_ast_Z_I}.

\begin{corollary}\label{c:Zeroes_I_ast_Zeroes_I}
Let $I$ be an ideal of $\lone$. Then $\Zeroes(I)=\Zeroes(I^\ast)$.
\end{corollary}

\begin{proof}
We may assume that $I$ is closed. Part (4) of Theorem~\ref{t:hull_kernel_setup_L_1_model} furnishes
 a self-adjoint ideal $I^\prime\supset I$ such that $\Zeroes(I^\prime)=\Zeroes(I)$ (where $I^\prime=\lone$ if $\Zeroes(I)=\emptyset$). Since $I^\prime\supset I^\ast$, we have $\Zeroes(I^\ast)\supset\Zeroes(I^\prime)=\Zeroes(I)$. Likewise, $\Zeroes(I^\ast)\supset\Zeroes(I^{\ast\ast})=\Zeroes(I)$.
\end{proof}

It is now possible to give conditions equivalent to spectral synthesis holding in this model. Of course, Theorem~\ref{t:fixed_points_of_IZ} is then applicable to all closed ideals of $\lone$. Moreover, since the freeness of $\dynsysshort$ is one of the conditions, all equivalent conditions of Theorem~\ref{t:all_ideals_well_behaved} are also valid, and Theorem~\ref{t:ideal_well_behaved} is applicable to all closed ideals.

\begin{theorem}\label{t:all_ideals_kernels}
The following are equivalent:
\begin{enumerate}
\item The maps $I \rightarrow \Zeroes(I)$ and $S \rightarrow \Ideal(S)$ are mutually inverse bijections between the set of closed ideals of $\lone$ on the one hand, and the subset $\{\Zeroes(I) : I \textup{ a closed ideal of }\lone\}$ of all $\prodhomeo$-invariant closed subsets of $\XtimesT$ on the other hand;
\item Every closed ideal of $\lone$ is the kernel of an involutive representation of $\lone$;
\item Every closed ideal of $\lone$ is self-adjoint;
\item Every closed ideal of $\lone$ is the intersection of primitive ideals;
\item $\dynsysshort$ is free.
\end{enumerate}
\end{theorem}

\begin{proof}
The equivalence of (1) and (2) is immediate from Theorem~\ref{t:hull_kernel_setup_L_1_model}. Certainly (2) implies (3), which by Theorem~\ref{t:all_ideals_well_behaved} is equivalent with (4) and (5). If (5) holds, then Theorem~\ref{t:all_ideals_well_behaved} shows that each closed ideal of $\lone$ is an intersection of kernels of involutive representations, hence is itself such a kernel. Thus (5) implies (2).
\end{proof}

\begin{remark}\label{r:similarity_to_C_ast_algebras}
The above result can be interpreted, as follows. The corresponding properties under (2), (3), and (4) in Theorem~\ref{t:all_ideals_kernels} are valid for all $C^*$-algebras. For the Banach algebras with isometric involution under consideration, these three properties are either all present or all absent, and they are all present precisely when the underlying dynamical system is free.
\end{remark}

\begin{remark}
For general locally compact abelian $G$, $L^1(G)$ is a regular commutative Banach algebra, i.e., every closed subset of $\widehat G$ is the hull of a closed ideal of $L^1(G)$. One might surmise that, in our case, the set $\Idealdomain$ in Theorem~\ref{t:all_ideals_kernels} consists of \emph{all} $\prodhomeo$-invariant closed subsets of $\XtimesT$. This is, however, not the case. If the system is free, then the combination of Theorem~\ref{t:all_ideals_kernels}, Theorem~\ref{t:fixed_points_of_IZ} and Proposition~\ref{p:zeroes_of_well_behaved_ideal} shows that $\Idealdomain=\{A\times\T : A\subset\topspace\textup{ $\homeo$-invariant and closed}\}$, and this set does not exhaust the $\prodhomeo$-invariant closed subsets of $\XtimesT$.
\end{remark}

\appendix

\section{Hulls and kernels: abstract framework}\label{sec:appendix}

In this Appendix, we collect some basic results on the general set-theoretical framework underlying hull-kernel-type constructions. Although the results and arguments are elementary and have been used in many particular cases, we are not aware of a general reference, and in view of their occurrence in both Section~\ref{sec:function_space_model} and Section~\ref{sec:l_one_model}, we find it worthwhile to make them explicit.

Let $A$ and $B$ be sets, supplied with a binary relation $\prec$ which is anti-symmetric, i.e., if $a_1,a_2\in A$, $a_1\prec a_2$ and $a_2\prec a_1$, then $a_1=a_2$, and likewise for $B$. We  use $a_2\succ a_1$ as an equivalent notation for $a_1\prec a_2$, and likewise for $B$. We do not assume $\prec$ to be reflexive or transitive. Furthermore, let $\amap: A\to B$ and $\bmap : B\to A$ be maps with the following properties.
\begin{assumption}\quad
\begin{enumerate}
 \item
\begin{enumerate}
\item $\bmap\circ\amap(a)\succ a$, for all $a\in A$;
\item if $a_1,a_2\in A$ and $a_1\succ a_2$, then $\amap(a_1)\prec \amap(a_2)$;
\end{enumerate}
\item \begin{enumerate}
\item $\amap\circ\bmap(b)\succ b$, for all $b\in B$;
\item if $b_1,b_2\in B$ and $b_1\succ b_2$, then $\bmap(b_1)\prec \bmap(b_2)$.
\end{enumerate}
\end{enumerate}
\end{assumption}

Thus there is full symmetry in $A$ and $\amap$ on the one hand, and $B$ and $\bmap$ on the other hand. Hence in the results below it would be sufficient to give just one of the statements, but it seems convenient for practical situations to formulate both. Naturally, we prove only one of them.

A typical example of this setup occurs when $A$ is the set of closed ideals of a commutative Banach algebra, and $B$ is the power set of its maximal ideal space, with $\prec$ denoting inclusion in both cases. If $I\in A$, then one lets $\amap(I)$ be the usual hull $\hull(I)$ of $I$, and if $S\in B$, then $\bmap(S)$ is the usual kernel $\kernel (S)$ of $S$. The fixed elements in $B$ of $\amap\circ\bmap=\hull\circ\kernel$ constitute the closed subsets in the hull-kernel topology on the maximal ideal space. The collection of such closed subsets coincides with the collection of hulls of closed ideals, (cf.\ part (2) of Corollary~\ref{c:corollary_of_three_maps_lemma}) and one of the main issues in spectral synthesis for commutative Banach algebras is the injectivity of the map $\amap=\hull$ on the set $A$ of closed ideals.

Likewise, the introduction of the Jacobson topology on the primitive ideal space of a general algebra falls within this framework, and the same holds true for the operations in Sections~\ref{sec:function_space_model} and \ref{sec:l_one_model}. In Section~\ref{sec:function_space_model} one takes for $A$ the set of ideals of $\lone$ (ordered by inclusion), with $\amap=\Hull$, and for $B$ the subsets of $\topspace$ (ordered by inclusion), with $\bmap=\Kernel$. In Section~\ref{sec:l_one_model} one takes for $A$ the set of ideals (or: closed ideals) of $\lone$ (ordered by inclusion) again, but now with $\amap=\Zeroes$, and for $B$ the subsets (or: closed subsets) of $\topspace\times \T$ invariant under $\homeo\times\idmap_\T$ (ordered by inclusion), with $\bmap=\Ideal$. That the above Assumption is then satisfied is the content of  Lemma~\ref{l:basic_Hull_properties}, Lemma~\ref{l:basic_Kernel_properties}, and Lemma~\ref{l:Hull_Kernel_composition} for Section~\ref{sec:function_space_model}, and of Lemma~\ref{l:basic_Zeroes_properties}, Lemma~\ref{l:basic_Ideal_properties} and Lemma~\ref{l:Zeroes_Ideal_composition} for Section~\ref{sec:l_one_model}.

\begin{lemma}\label{l:three_maps_lemma}\quad
$\amap\circ\bmap\circ\amap=\amap$ and $\bmap\circ\amap\circ\bmap=\bmap$.
\end{lemma}

\begin{proof}
Let $a\in A$. Then $\bmap\circ\amap(a)\succ a$ by part (1)(a) of the Assumption, hence part (1)(b) of the Assumption implies $\amap\circ\bmap\circ\amap(a)\prec \amap(a)$. On the other hand, part (2)(a) of the Assumption shows that $\amap\circ\bmap\circ\amap(a)=\amap\circ\bmap(\amap(a))\succ\amap(a)$. Hence we have equality.
\end{proof}

Corollary~\ref{c:corollary_of_three_maps_lemma} and Corollary~\ref{c:pre_image} are based only on the properties of $\amap$ and $\bmap$ in Lemma~\ref{l:three_maps_lemma}. We let $\fixab$ denote the fixed points in $B$ of $\amap\circ\bmap$, and similarly for $\fixba$.

\begin{corollary}\label{c:corollary_of_three_maps_lemma}\quad
\begin{enumerate}
\item $(\bmap\circ\amap)^2=\bmap\circ\amap$ and $(\amap\circ\bmap)^2=\amap\circ\bmap$.
\item $\amap(A)=\fixab$ and $\bmap(B)=\fixba$.
\item The restricted maps $\amap: \fixba\to\fixab$ and $\bmap:\fixab\to\fixba$ are mutually inverse bijections.
\item
\begin{enumerate}
\item The following are equivalent:
\begin{enumerate}
\item $\amap$ is injective on $A$;
\item $A=\bmap(B)$;
\item $A=\fixba$;
\item $\{a\in A : \amap(a)=b\}=\{\bmap(b)\}$, for all $b\in\fixab$.
\end{enumerate}
\item The following are equivalent:
\begin{enumerate}
\item $\bmap$ is injective on $B$;
\item $B=\amap(A)$;
\item $B=\fixab$;
\item $\{b\in B : \bmap(b)=a\}=\{\amap(a)\}$, for all $a\in\fixba$.
\end{enumerate}
\end{enumerate}
\end{enumerate}
\end{corollary}

\begin{proof}
Part (1) is immediate from Lemma~\ref{l:three_maps_lemma}. As to part (2), if $b\in\amap(A)$, say $b=\amap(a)$ for $a\in A$, then $\amap\circ\bmap(b)=(\amap\circ\bmap\circ\amap)(a)=\amap(a)=b$ by Lemma~\ref{l:three_maps_lemma}. Hence $\amap(A)\subset\fix(\amap\circ\bmap)$. Since the reverse inclusion is obvious, we have equality. For part (3), we need only remark that the codomains are appropriate as a consequence of part (2), since it is then obvious that the restricted maps are mutually inverse bijections.
The parts (2) and (3) yield $\amap(A)=\fixab=\amap(\fixba)$, and this implies the equivalence in part (4)(a).
\end{proof}

The picture to keep in mind is the following.

\begin{tikzpicture}[scale=0.8]
\draw [black, thick] (0,0) circle [radius=3.5];
\draw [fill=gray!27][style=thick](0,0) circle [radius=2.5];
\draw [black, thick] (8,0) circle [radius=3.5];
\draw [fill=gray!27][style=thick] (8,0) circle [radius=2.5];
\draw[->, thick] (0,3) to [out=45,in=135] (8,1.5);
\draw[->, thick] (8,-3) to [out=-135,in=-45] (0,-1.5);
\draw[->, thick] (1,0.1) to [out=30,in=150](7,0.1);
\draw[->, thick] (7,-0.1) to [out=-159,in=-30](1,-0.1);
\node at (4,3.5) {$\mathbf \alpha$};
\node at (4,-3.5) {$\mathbf \beta$};
\node at (4,1.3) {$\mathbf \alpha$};
\node at (4,0.65){$\simeq$};
\node at (4,-1.3) {$\mathbf \beta$};
\node at (4,-0.55){$\simeq$};
\node at (-3,3){$A$};
\node at (11,3){$B$};
\node at (-0.6,0.55) {$\beta(B)$};
\node at (-0.6,0){$=$};
\node at (-0.6,-0.55){$\textup{Fix}(\beta\circ\alpha)$};
\node at (8.6,0.55) {$\alpha(A)$};
\node at (8.6,0){$=$};
\node at (8.6,-0.55){$\textup{Fix}(\alpha\circ\beta)$};
\end{tikzpicture}

The following is now clear.

\begin{corollary}\label{c:pre_image}
If $a\in A$, then $\{a^\prime\in \fixba : \amap(a^\prime)=\amap(a)\}=\{\bmap\circ\amap(a)\}$.
\\
If $b\in B$, then $\{b^\prime\in \fixab : \bmap(b^\prime)=\bmap(b)\}=\{\amap\circ\bmap(b)\}$.
\end{corollary}

\begin{lemma}\label{l:min_max_lemma}\quad
\begin{enumerate}
\item Let $a\in A$. Then
\begin{align*}
\bmap\circ\amap(a)&=\min\{a^\prime\in\fixba : a^\prime\succ a\}\\
&=\max\{a^\prime\in A: \amap(a^\prime)=\amap(a)\}.
\end{align*}
\item Let $b\in B$. Then
\begin{align*}
\amap\circ\bmap(b)&=\min\{b^\prime\in\fixab : b^\prime\succ b\}\\
&=\max\{b^\prime\in B: \bmap(b^\prime)=\bmap(b)\}.
\end{align*}
\end{enumerate}
\end{lemma}

\begin{proof}
Let $a\in A$, and put $S_1=\{a^\prime\in\fixba : a^\prime\succ a\}$. From part (1)(a) of the Assumption we have $\bmap\circ\amap (a)\succ a$. Since furthermore $\bmap\circ\amap(a)\in\bmap(B)=\fixba$, we see that $\bmap\circ\amap(a)\in S_1$. If $a^\prime\in S_1$, then $a^\prime\succ a$ implies $\amap(a^\prime)\prec\amap(a)$, hence $\bmap\circ\amap(a^\prime)\succ\bmap\circ\amap(a)$. Since additionally $a^\prime\in\fixba$, we see that $a^\prime\succ\bmap\circ\amap(a)$. Hence $\bmap\circ\amap$ is the (automatically unique) smallest element of $S_1$, as required. Turning to the second equality, let $S_2=\{a^\prime\in A: \amap(a^\prime)=\amap(a)\}$. Since $\amap(\bmap\circ\amap(a))=\amap(a)$ by Lemma~\ref{l:three_maps_lemma}, we see that $\bmap\circ\amap(a)\in S_2$. If $a^\prime\in S_2$, then $\amap(a^\prime)=\amap(a)$ implies $\bmap\circ\amap(a)=\bmap\circ\amap(a^\prime)\succ a^\prime$. Hence $\bmap\circ\amap(a)$ is the (automatically unique) largest element of $S_2$, as required.
\end{proof}

\begin{lemma}\label{l:order_inflection}
Let $a\in A$. Then the following are equivalent:
\begin{enumerate}
\item For all $a^\prime\in A$, $a^\prime\prec a$ if and only if $\amap(a^\prime)\succ \amap(a)$;
\item $a\in\fixba$;
\item $a\in\bmap(B)$.
\end{enumerate}
Let $b\in B$. Then the following are equivalent:
\begin{enumerate}
\item[(4)] For all $b^\prime\in B$, $b^\prime\prec b$ if and only if $\bmap(b^\prime)\succ \bmap(b)$;
\item[(5)] $b\in\fixab$;
\item[(6)] $b\in\amap(A)$.
\end{enumerate}
\end{lemma}

\begin{proof}
We prove the statement for $A$. Suppose (1) holds. Since $\amap(\bmap\circ\amap(a))=\amap(a)$, we then have $\bmap\circ \amap(a)\prec a$. As always $\bmap\circ \amap(a)\succ a$, we have equality. Hence (1) implies (2). Assume that (2) holds, and suppose $a^\prime\in A$. Certainly $a^\prime\prec a$ implies $\amap(a^\prime)\succ \amap(a)$. If $\amap(a^\prime)\succ \amap(a)$, then $\bmap\circ\amap(a^\prime)\prec \bmap\circ\amap(a)=a$. Since $\bmap\circ\amap(a^\prime)\succ a^\prime$, we have $a^\prime\prec a$. Hence (2) implies (1). The equivalence of the second and third part has already been noted in Corollary~\ref{c:corollary_of_three_maps_lemma}.
\end{proof}

\subsection*{Acknowledgements}
This work was supported by visitor's grants of the Netherlands Organisation for Scientific Research (NWO) and the Dutch research cluster Geometry and Quantum Theory (GQT).

%%%%%%%%%%%%%%%%%%%%%%%%%%%%%%%%%%%%% END OF ACTUAL TEXT %%%%%%%%%%%%%%%%%%%%%%%%%%%%%%%

\end{document}